\documentclass[a4paper,12pt,reqno]{amsart}

\usepackage{amsthm}
\usepackage{amsmath}
\usepackage{amssymb}
\usepackage{mathrsfs}
\usepackage{latexsym}
\usepackage{exscale}
\usepackage{geometry} 
\usepackage{graphicx} 
\usepackage[utf8]{inputenc}
\usepackage{verbatim}
\usepackage{enumerate} 
\usepackage{ esint } 
\usepackage{fancyhdr}

\usepackage{color}
\definecolor{red}{rgb}{1,0.1,0.1}
\definecolor{blue}{rgb}{0.1,0.1,1}
\definecolor{vb}{RGB}{160,32,240}

\usepackage[colorlinks=true, pdfstartview=FitV, linkcolor=blue, citecolor=red, urlcolor=blue]{hyperref}

\theoremstyle{plain}
\DeclareMathOperator{\esssup}{ess sup}

\numberwithin{equation}{section}
\newtheorem{teo}{Theorem}[section]
\newtheorem{lemma}[teo]{Lemma}
\newtheorem{cor}[teo]{Corollary}
\newtheorem{prop}[teo]{Proposition}

\theoremstyle{remark}
\newtheorem{remark}[teo]{Remark}

\theoremstyle{definition}

\newtheorem*{mydef*}{Definition}

\calclayout

\allowdisplaybreaks

\newcommand{\R}{\mathbb{R}^n}

\begin{document}

\title[Mixed weak type inequalities]{Mixed weak type inequalities in euclidean spaces and in spaces of homogeneous type}

\author[Gonzalo~Iba\~{n}ez-Firnkorn]{Gonzalo Iba\~{n}ez-Firnkorn}
\address{(Gonzalo~Iba\~{n}ez-Firnkorn) INMABB (CONICET) \\ Bahía Blanca, Argentina}
\email{gonzaibafirn@gmail.com}

\author{Israel P. Rivera-Ríos}
\address{(Israel P. Rivera-Ríos) Departamento de Análisis Matemático, Estadística e Investigación Operativa
y Matemática Aplicada. Facultad de Ciencias. Universidad de Málaga
\\ Málaga, Spain. Departamento de Matemática. Universidad Nacional
del Sur \\ Bahía Blanca, Argentina}
\email{israelpriverarios@uma.es}

\thanks{The first author was partially supported by CONICET and SECYT-UNC.
The second author was partially supported by FONCyT PICT 2018-02501 and PICT 2019-00018 and by Junta de Andaluc\'ia UMA18FEDERJA002.}



\begin{abstract}
In this paper we provide mixed weak type inequalities generalizing previous results in an earlier work by Caldarelli and the second author \cite{CRR20} and also in the spirit of earlier results by Lorente, Martell, P\'erez and Riveros \cite{LMPR}. One of the main novelties is that, besides obtaining estimates in the euclidean setting, results are provided as well in spaces of homogeneous type, being the first mixed weak type estimates that we are aware of in that setting.\end{abstract}

\maketitle
\section{Introduction}

In \cite{MW77}, Muckenhoupt and Wheeden, introduced a new type of weak type inequality, that consists in considering a perturbation of the Hardy-Littlewood maximal operator, $M$, with an $A_p$ weight. The result was the following, if $w\in A_1$ then
\begin{equation*} 
\left|\left\{ x \in \mathbb{R}^n : w(x)Mf(x)>\lambda\right\}\right|
\leq c_{w} \int_{\mathbb{R}^n} \frac{|f(x)|}{\lambda} w(x)dx.
\end{equation*}

Inequalities with this kind of perturbation, and some further ones that we describe in the following lines, are known  in the literature as mixed weak-type inequalities. 

In \cite{S85}, having as an application an alternative proof to Muckenhoupt's $A_p$ theorem, Sawyer settled the following result. If $u,v\in A_1$ then

\begin{equation}
uv\left(\left\{ x \in \mathbb{R}^n : \frac{M(fv)(x)}{v(x)}>\lambda\right\}\right)
\leq  c_{u,v} \int_{\mathbb{R}^n} \frac{|f(x)|}{\lambda} u(x)v(x)dx.
\end{equation}

A further extension of this inequality was provided in \cite{CUMP05}. More precisely it was shown that is $u$ and $v$ satisfy either $u,v\in A_1$ or $u\in A_1$ and $v\in A_{\infty}(u)$, then the inequality
\begin{equation}
uv\left(\left\{ x \in \mathbb{R}^n : \frac{\mathcal{T}(fv)(x)}{v(x)}>\lambda\right\}\right)
\leq c_n c_{u,v} \int_{\mathbb{R}^n} \frac{|f(x)|}{\lambda} u(x)v(x)dx
\end{equation}
holds for every $t$ where $\mathcal{T}$ is either $M$ or a Calderón-Zygmund operator. Note that the assumption $u,v\in A_1$ is weaker, in the sense that the product $uv$ does not necessarily have any regularity, in contrast to what happens with the assumption $u\in A_1$ and $v\in A_\infty(u)$ for which $uv\in A_\infty$. It was also conjectured in \cite{CUMP05} that the assumption $u,v\in A_1$ could be further weakened to $u\in A_1$ and $v\in A_\infty$. That conjecture was solved in the positive in \cite{LOP19}.

Over the past few years, there have been some contributions on mixed weak-type inequalities such as \cite{BCP19} for the case of fractional integral and related operators, \cite{B19} for generalized maximal functions, \cite{OP16, OPR16, CRR20} for related quantitative estimates and \cite{LOPi19} for multilinear extensions. In \cite{BCP19} results for the commutator of Calder\'on Zygmund operators and H\"ormander type operators were proved and in \cite{CRR20} the authors provided quantitative estimates for commutators of Calder\'on-Zygmund operators and rough singular integral operators via sparse domination. 

In this paper we aim to provide some mixed weak type estimates both in the euclidean setting and in spaces of homogeneous type being the results in the latter setting the first ones in such a degree of generality that we are aware of. 
Also our results can be regarded as a  revisit of certain results in \cite{LMPR}.
In both cases, our approach will rely upon sparse domination, pushing forward ideas in  \cite{CRR20}. We describe our contribution in the following subsections. 

\subsection{Results in the euclidean setting}

Our first result is concerned with operators satisfying a bilinear sparse domination result.
Given $A$ a Young function (see Subsection \ref{subsec:Young} for the precise definition), we assume that $T$ admits the following bilinear sparse bound
\begin{equation}\label{bilinearsparseT}
\int T(f)g \leq c_n \sum_{j=1}^{3^n}\sum_{Q\in \mathcal{S}_j} \|f\|_{1,Q}\|g\|_{A,Q}|Q|
\end{equation}
where $\mathcal{S}_j$ are dyadic sparse families.

\begin{teo}\label{MixtaT} Let $1\leq p, r <\infty$. Let $u\in A_1\cap RH_{q}$ with $q=2r-1$  and $v\in A_p(u)$. Let $A$ be a Young function such that  $A \in B_{\rho}$ for all $\rho> r$ and $T$ be an operator that satisfies \eqref{bilinearsparseT}.
Then,

$$uv\left(\left\{x \in \R : \frac{T(fv)(x)}{v(x)}>\lambda\right\}\right)\leq c_{n,T} C_{u,v} \int_{\R} \frac{|f(x)|}{\lambda} u(x)v(x) dx, $$
where 
\begin{align*}
C_{uv} =  \kappa_{u} [u]_{RH_q}^{1+\frac{q}{4r}} [u]_{A_1}[uv]_{A_{\infty}}\log(e+\kappa_{u} [u]_{RH_q}^{1+\frac{q}{4r}} [u]_{A_1}[uv]_{A_{\infty}}[v]_{A_p(u)})
\end{align*} 

In the case of $A(t)=t^{r}(1 + \log^+t)^{\gamma}$ we have
$\kappa_u=[u^r]_{A_{\infty}}^{\frac{\gamma}{r}}\leq [u]_{RH_q}^{\frac{\gamma}{r}}[u]_{A_1}^{\frac{\gamma}{r}}$. 
\end{teo}

Before presenting our next result we need some notation. Let ${\bf b}=(b_1,b_2, \dots, b_m)$ be a set of symbols with $b_i\in \text{Osc}_{\exp L^{r_i}}$, $i=1,\dots, m$. Let ${\bf b}=\sigma \cup \sigma'$ where $\sigma$ and $\sigma'$ are pairwise disjoint sets. We introduce the following notation
\[\begin{split}(b - {\bf{\lambda}})_{\sigma}&= \prod_{i\in \sigma} (b_i(x)-\lambda_i)\\
|b - {\bf{\lambda}}|_{\sigma}&= \prod_{i\in \sigma} |b_i(x)-\lambda_i|
\end{split}\]
where ${\bf \lambda}=(\lambda_1,\lambda_2, \dots, \lambda_m)$. 
Let $C_j(b)$ the family of all the subsets $\sigma$ of $b$ such that $\# \sigma =j$.  Observe that the cardinality of $C_j(b)$ is ${m}\choose{j}$. Also we denote 
$$\|{\bf b}\| = \prod_{i=1}^m \|b_i\|_{\text{Osc}_{\exp L^{r_i}}}.$$
We remit the reader to Subsection \ref{subsec:Young} for the definition of the spaces $\text{Osc}_{\exp L^{r_i}}$.

Having the notation above at our disposal we define $T_{\bf b}$ as follows
\[T_{\bf b}f(x)=[b_m,\dots[b_2,[b_1,T]]]f(x)\]
At this point we are in the position to present our next result. Let $A$ be a Young function. We consider $T$ an operator such that $T_{\bf b}$ satisfies the following bilinear sparse bound
\begin{equation}\label{bilinearsparseTb}
\int T_{\bf b}(f)g \leq c_n \sum_{j=1}^{3^n}\sum_{h=0}^m \sum_{\sigma \in C_h(b)}\sum_{Q\in \mathcal{S}_j} \|f|b-b_Q|_{\sigma}\|_{1,Q}\|g|b(x)-b_Q|_{\sigma'}\|_{A,Q}|Q|
\end{equation}
where $\mathcal{S}_j$ are dyadic sparse families. For such operators we have the following result.
\begin{teo}\label{multiSimb} Let $1\leq p, s <\infty$. Let $u\in A_1\cap RH_{q}$  with $q=2s-1$ and $v\in A_p(u)$. 
Let $m\in \mathbb{N}$, $r_i\geq 1$ for every $1\leq i \leq m$, $\frac1{r}=\sum_{i=1}^m\frac1{r_i}$ and ${\bf b}=(b_1,\dots, b_m)$ where $b_i\in \text{Osc}_{\exp L^{r_i}}$ for $1\leq i \leq m$.
Let $A$ and $B$ be Young functions such that $B^{-1}(t)\log (t)^{1/r}\lesssim {A}^{-1}(t)$ for all $t\geq e$ and  $B \in B_{\rho}$ for all $\rho > s$ and $T$ be an operator that satisfies \eqref{bilinearsparseTb}. 
  Then, 
$$uv\left(\left\{x \in \R : \frac{T_{\bf b}(fv)(x)}{v(x)}>\lambda\right\}\right)\leq c_{n,T} C_{u,v} \int_{\R} \varphi_{\frac{1}{r}}\left(\frac{|f(x)|\| {\bf b}\|}{\lambda}\right) u(x)v(x) dx, $$
where $\varphi_{\frac{1}{r}}(t)=t(1 + \log^+t)^{\frac{1}{r}}$  and
\begin{align*}
C_{uv} =  \sum_{h=0}^m \kappa_u [u]_{RH_{q}}^{1+\frac{q}{4s}} [u]_{A_1} [uv]_{A_{\infty}}^{1+\frac{m-h}{r}}\log(e+[uv]_{A_{\infty}}^{1+\frac{m-h}{r}} \kappa_u[u]_{RH_{q}}^{1+\frac{q}{4s}}[u]_{A_1}[v]_{A_p(u)})^{1+\frac1{r}}
\end{align*} 
In the case of $A(t)=t^{s}(1 + \log^+t)^{\gamma}$ we have
$\kappa_u=[u^s]_{A_{\infty}}^{\frac{\gamma}{s}}\leq [u]_{RH_q}^{\frac{\gamma}{s}}[u]_{A_1}^{\frac{\gamma}{s}}$.  
\end{teo}

The usual iterated commutator is just a particular case of the result above that consists just in assuming that all the symbols involved coincide. Let $A$ be a Young function, $m$ be a positive integer and $ b\in BMO$. We consider $T$ an operator such that the $m$-order iterated commutator $T_{b}^m$ satisfies the following bilinear sparse bound
\begin{equation}\label{bilinearsparsecomm}
\int T_b^m(f)g \leq c_n \sum_{j=1}^{3^n} \sum_{Q\in \mathcal{S}_j} \biggl[
 \||b-b_Q|^mfv \|_{1,Q}  \|gu\|_{A,Q} |Q|+   \|fv\|_{1,Q}\|gu|b-b_Q|^m\|_{A,Q} |Q| \biggr]
\end{equation}
where $\mathcal{S}_j$ are dyadic sparse families. Then we have the following result.
\begin{teo}\label{conmT} Let $1\leq p,s <\infty$. Let $u\in A_1\cap RH_{q}$  with $q=2s-1$ and $v\in A_p(u)$. 
Let $m\in \mathbb{N}$ and $b\in BMO$. 
Let $A$ and $B$ be  Young functions such that  $B^{-1}(t)\log (t)^{m}\lesssim {A}^{-1}(t)$ for all $t\geq e$ and  $B \in B_{\rho}$ for all $\rho > s$.
Let $T$ be an operator  that satisfies \eqref{bilinearsparsecomm}. Then,  
$$uv\left(\left\{x \in \R : \frac{T_{ b}^m(fv)(x)}{v(x)}>\lambda\right\}\right)\leq c_{n,T} C_{u,v} \int_{\R} \varphi_{m}\left(\frac{|f(x)|\|  b\|_{BMO}^m}{\lambda}\right) u(x)v(x) dx, $$
where $\varphi_{m}(t)=t(1 + \log^+t)^{m}$ and 

\begin{align*}
C_{uv} = &[u]_{RH_{q}}^{1+\frac{q}{4s}}\kappa_u[u]_{A_1} [uv]_{A_{\infty}}^{1+m}\log(e+ \kappa_u[u]_{A_1}[uv]_{A_{\infty}}^{1+m}[v]_{A_p(u)})^{1+m}
\\& +  [u]_{RH_{q}}^{1+\frac{q}{4s}}\kappa_u[u]_{A_1} [uv]_{A_{\infty}}\log(e+ \kappa_u[u]_{A_1}[uv]_{A_{\infty}}[v]_{A_p(u)}).
\end{align*}
\end{teo}

\subsection{Results in spaces of homogeneous type}
We recall that $(X,d,\mu)$ is a space of homogeneous type if $X$ is
a set endowed with a quasi-metric $d$ and a doubling Borel measure
$\mu$. $d$ is a quasi-metric if there exists a constant
$\kappa_{d}\geq1$ such that
\[
d(x,y)\leq \kappa_d\left(d(x,z)+d(z,y)\right)\qquad x,y,z\in X,
\]
namely, if the triangle inequality holds modulo a constant.
Since $\mu$ satisfies the doubling property, we have that there exists $c_{\mu}\geq1$
such that 
\[
\mu(B(x,2\rho))\leq c_{\mu}\mu(B(x,\rho))\qquad x\in X,\quad\rho>0
\]
where $B(x,\rho):=\{y\in X\,:\,d(x,y)<\rho\}$. We will assume additionally
that all balls $B$ are Borel sets and that $0<\mu(B)<\infty$. 

Since $\mu$ is a Borel measure defined on the Borel $\sigma$-algebra
of the quasi-metric space $(X,d)$ we have that the Lebesgue differentiation theorem
holds. This yields that continuous functions with bounded support are dense
in $L^{p}(X)$ for every $1\leq p<\infty$. 

Let $A$ be a Young function. Let  $T$ be an operator satisfying the following bilinear sparse bound
\begin{equation}
\int T(f)g\leq c\frac{1}{1-\varepsilon}\sum_{j=1}^{l}\sum_{Q\in\mathcal{S}_{j}}\|f\|_{1,Q}\|g\|_{A,Q}\mu(Q)\label{bilinearsparseTETH}
\end{equation}
where $\mathcal{S}_{j}$ are $\varepsilon$-sparse families of dyadic cubes (see Subsection \ref{sec:dyadicSt} for the precise definition). Observe that in contrast with the results in the euclidean case, in this setting we need to be precise about the sparseness constant of the families involved since it plays a role in the proof. 
\begin{teo}\label{thm:Teth} Let $1\leq p,r<\infty$. Let $u\in A_{1}\cap RH_{q}$
with $q=2r-1$ and $v\in A_{p}(u)$. Let $A$ be a Young function
such that $A\in B_{\rho}$ for all $\rho>r$ and let $T$ be an operator
that satisfies (\ref{bilinearsparseTETH}). Then,
\[
uv\left(\left\{ x\in X:\frac{T(fv)(x)}{v(x)}>\lambda\right\} \right)\leq c_{X,T}C_{u,v}\int_{X}\frac{|f(x)|}{\lambda}u(x)v(x)d\mu(x),
\]
where
\begin{align*}
C_{uv}=\kappa_{u}[u]_{RH_{q}}^{1+\frac{q}{4r}}[u]_{A_{1}}[uv]_{A_{p}}[uv]_{A_{\infty}}\log(e+c_{n,p}\kappa_{u}[u]_{RH_{q}}^{1+\frac{q}{4r}}[u]_{A_{1}}[v]_{A_{p}(u)}[uv]_{A_{p}}^{3}) & .
\end{align*}
In the case of $A(t)=t^{r}(1+\log^{+}t)^{\gamma}$ we have $\kappa_{u}=[u^{r}]_{A_{\infty}}^{\frac{\gamma}{r}}\leq[u]_{RH_{q}}^{\frac{\gamma}{r}}[u]_{A_{1}}^{\frac{\gamma}{r}}$. 
\end{teo}

In the case of iterated commutators we need to assume that
\begin{equation}\label{bilinearsparseCommETH}
\int T_{{\bf b}}(f)g\leq c\frac{1}{1-\varepsilon}\sum_{j=1}^{l}\sum_{h=0}^{m}\sum_{\sigma\in C_{h}(b)}\sum_{Q\in\mathcal{S}_{j}}\left(\int_{Q}f|b-b_{Q}|_{\sigma'}d\mu\right)\|g|b-b_{Q}|_{\sigma}\|_{A,Q}
\end{equation}
where $\mathcal{S}_{j}$ are $\varepsilon$-sparse families of dyadic cubes (see Subsection \ref{sec:dyadicSt} for the precise definition). Under this assumption we have the following result.
\begin{teo}\label{thm:bTeth} Let $1\leq p,s<\infty$. Let $u\in A_{1}\cap RH_{q}$
with $q=2s-1$ and $v\in A_{p}(u)$. Let $m\in\mathbb{N}$, $r_{i}\geq1$
for every $1\leq i\leq m$, $\frac{1}{r}=\sum_{i=1}^{m}\frac{1}{r_{i}}$
and ${\bf b}=(b_{1},\dots,b_{m})$ where $b_{i}\in\text{Osc}_{\exp L^{r_{i}}}$
for $1\leq i\leq m$. Let $A$ and $B$ be Young functions such that
$B^{-1}(t)\log(t)^{1/r}\lesssim A^{-1}(t)$ for all $t\geq e$
and $B\in B_{\rho}$ for all $\rho>s$ and $T$ be an operator that
satisfies (\ref{bilinearsparseCommETH}). Then, 
\[
uv\left(\left\{x\in X:\frac{T_{{\bf b}}(fv)(x)}{v(x)}>\lambda\right\}\right)\leq c_{n,T}C_{u,v}\int_{X}\varphi_{\frac{1}{r}}\left(\frac{|f(x)|\|{\bf b}\|}{\lambda}\right)u(x)v(x)dx,
\]
where $\varphi_{\frac{1}{r}}(t)=t(1+\log^{+}t)^{\frac{1}{r}}$ and
\begin{align*}
C_{u,v}=\sum_{h=0}^{m}\tau_{u,v,h}[uv]_{A_{\infty}}\log(e+\tau_{u,v,h}[uv]_{A_{p}}^{2}[v]_{A_{p}(u)})^{1+\frac{1}{r}}
\end{align*}
where $\tau_{u,v,h}=\kappa_{u}[u]_{RH_{q}}^{1+\frac{q}{4s}}[u]_{A_{1}}\left([uv]_{A_{p}}[uv]_{A_{\infty}}\right)^{\frac{m-h}{r}}[uv]_{A_{p}}$.
In the case $A(t)=t^{s}(1+\log^{+}t)^{\gamma}$, additionally we have
that $\kappa_{u}=[u^{s}]_{A_{\infty}}^{\frac{\gamma}{s}}\leq[u]_{RH_{q}}^{\frac{\gamma}{s}}[u]_{A_{1}}^{\frac{\gamma}{s}}$. 
\end{teo}

\begin{remark}At this point we would like to note that the dependences obtained in the case of spaces of homogeneous type are slightly worse than those in the euclidean setting. The additional constants appear due to the fact that reverse H\"older inequality is actually a weak reverse H\"older inequality for balls instead of cubes in this setting, and due to the fact that doubling conditions do not behave as good as in the euclidean setting.\end{remark}

The remainder of the paper is organized as follows. In Section \ref{sec:prelim} we provide some preliminaries. Section \ref{sec:proofmain} is devoted to the proofs of the main results. Finally in Section \ref{sec:applications} we show how to derive results for $A$-H\"ormander operators and their commutators from the main results and  we provide a sparse domination result for $T_{\bf b}$ in the context of spaces of homogeneous type, generalizing \cite{DGKLW,HLO}.

\section{Preliminaries}\label{sec:prelim}

\subsection{Dyadic structures on spaces of homogeneous type}\label{sec:dyadicSt}

We shall follow the presentation and the notation provided in \cite{L21}.
Let us fix $0<c_{0}\leq C_{0}<\infty$ and $\delta\in(0,1)$. Assume
that for each $k\in\mathbb{Z}$ we have an index set $J_{k}$ and
a pairwise disjoint collection $\mathcal{D}_{k}=\{Q_{j}^{k}\}_{j\in J_{k}}$
of measurable sets and an associated collection of points $\{z_{j}^{k}\}_{j\in J_{k}}$.
We will say that $\mathcal{D}=\bigcup_{k\in\mathbb{Z}}\mathcal{D}_{k}$
is a dyadic system with parameters $c_{0},C_{0}$ and $\delta$ if
the following properties hold.
\begin{enumerate}
\item For every $k\in\mathbb{Z}$ 
\[
X=\bigcup_{j\in J_{k}}Q_{j}^{k}.
\]
\item For $k\geq l$ if $P\in\mathcal{D}_{k}$ and $Q\in\mathcal{D}_{l}$
then either $Q\cap P=\emptyset$ or $P\subseteq Q$. 
\item For each $k\in\mathbb{Z}$ and $j\in J_{k}$ 
\[
B(z_{j}^{k},c_{0}\delta^{k})\subseteq Q_{j}^{k}\subseteq B(z_{j}^{k},C_{0}\delta^{k}).
\]
\end{enumerate}
We will call the elements of $\mathcal{D}$ cubes and we will denote
\[
\mathcal{D}(Q):=\left\{ P\in\mathcal{D}:P\subseteq Q\right\} 
\]
the family of cubes of $\mathcal{D}$ that are contained in $Q$.
We will say, as well, that an estimate depends on $\mathcal{D}$ if
it depends on the parameters $c_{0}$, $C_{0}$ and $\delta$.

The point $z_{j}^{k}$ could be regarded as the ``center'' and $\delta^{k}$
as the ``side length'' of each cube $Q_{j}^{k}\in\mathcal{D}_{k}$.
These need to be with respect a certain $k\in\mathbb{Z}$ since $k$
may not be unique. Consequently, a cube $Q$ also encodes the information
of its center $z$ and generation $k$. 

We define the dilations $\alpha Q$ for $\alpha\geq1$ of $Q\in\mathcal{D}$
as 
\[
\alpha Q:=B(z_{j}^{k},\alpha C_{0}\delta^{k}).
\]
Abusing of this dilation notation we denote 
\[
1Q:=B(z_{j}^{k},C_{0}\delta^{k})
\]
Note that these dilations are not cubes anymore but balls.

The following proposition, settled in \cite{HK12}, ensures the existence
of dyadic systems that provide a convenient replacement for the translations
of the usual dyadic systems in euclidean spaces. 
\begin{prop}
\label{proposition:dyadicsystem} Let $(X,d,\mu)$ be a space of homogeneous
type. There exist $0<c_{0}\leq C_{0}<\infty$, $\gamma\geq1$, $0<\delta<1$
and $m\in\mathbb{N}$ such that there are dyadic systems $\mathcal{D}_{1},\dots,\mathcal{D}_{m}$
with parameters $c_{0}$, $C_{0}$ and $\delta$, and with the property
that for each $s\in X$ and $\rho>0$ there is a $j\in\left\{ 1,\cdots,m\right\} $
and a $Q\in\mathcal{D}_{j}$ such that 
\[
B(s,\rho)\subseteq Q,\qquad\text{and}\qquad\text{diam}(Q)\leq\gamma\rho.
\]
\end{prop}

We end up this section borrowing from \cite{L21} the following covering
Lemma. 
\begin{lemma}
\label{lemma:covering} Let $(X,d,\mu)$ be a space of homogeneous
type and $\mathcal{D}$ a dyadic system with parameters $c_{0}$,
$C_{0}$ and $\delta$. Suppose that $\text{diam}(X)=\infty$, take
$\alpha\geq3c_{d}^{2}/\delta$ and let $E\subseteq X$ satisfy $0<\text{diam}(E)<\infty$.
Then there exists a partition $\mathcal{P}\subseteq\mathcal{D}$ of
$X$ such that $E\subseteq\alpha Q$ for all $Q\in\mathcal{P}$. 
\end{lemma}

\subsection{Weights}
We recall that given a weight $u$,  $v\in A_p(u)$ if
$$[v]_{A_p(u)}=\sup_Q \frac1{u(Q)}\int_Q vu \left(\frac1{u(Q)}\int_Q v^{-\frac1{p-1}}u\right)^{p-1}<\infty,$$
in the case $1<p<\infty$ and 
$$[v]_{A_1(u)}=\left\|\frac{M_uv}{v} \right\|_{\infty}<\infty$$
where $M_uv=\sup_Q \frac1{u(Q)}\int_Q vu$. If $u=1$ we recover the classical Muckenhoupt's condition.
We would like also to recall that 
$$A_{\infty}=\bigcup_{p\geq 1}A_p$$
with the constant 
$$[w]_{A_{\infty}}=\sup_Q \frac1{w(Q)}\int_Q M(\chi_Qw)<\infty.$$

Also we recall the Reverse-H\"older's condition, $w\in RH_s$, $1<s<\infty$ if 
$$[w]_{RH_s}=\sup_Q \frac{\left(\frac1{|Q|}\int_Q w^s\right)^\frac1{s}}{\frac1{|Q|}w(Q)}<\infty.$$ 
If $s=1$ the condition $RH_1$ is trivial.

An auxiliary lemma that we need is the following
\begin{lemma}
Let $1\leq s \leq q <\infty$. If $u\in A_1 \cap RH_q$ then $u^s\in A_1$ and $[u^s]_{A_1}\leq [u]_{RH_q}[u]_{A_1}$
\end{lemma}

In spaces of homogeneous type we recall that the $A_p$ classes and the $A_p$ and $A_\infty$ constants are defined exactly in the same way as we showed above for the euclidean setting just replacing cubes by balls. There is an important difference with the euclidean setting for the reverse H\"older classes. We say that $w\in RH_q$ if
\[
\left(\frac{w^{q}(B)}{\mu(B)}\right)^{\frac{1}{q}}\leq[w]_{RH_{q}}\frac{w(c_{d}B)}{\mu(c_{d}B)}
\]
where $c_d$ is some constant depending on the space.
Another fundamental result for us will be the sharp reverse H\"older inequality that was settled in \cite{HPR12}. There it was established that if $w\in A_{\infty}$, then
\[
\left(\frac{u^{t}(B)}{\mu(B)}\right)^{\frac{1}{t}}\leq c\frac{u(c_{d}B)}{\mu(c_{d}B)}
\]
where $1\leq t\leq1+\frac{1}{\tau[u]_{A_{\infty}}}$  and $c,\tau>0$ are some constants depending on the space.
Note that from this property it follows as well that if $E\subset B$ then
\[w(E)\leq c\left(\frac{\mu(E)}{\mu(c_{d}B)}\right)^{\frac{1}{c_{X}[w]_{\infty}}}w(c_{d}B)\]
for some constants $c,c_X$ depending on the space $X$.

We continue with the following sum property.
\begin{lemma}
\label{CRRlema3Hom} Let $w\in A_{p}$ and $\mathcal{S}$ be a $\eta$-sparse
family of cubes. Then 
\[
\sum_{Q\in\mathcal{S}}w(Q)\leq c\frac{1}{\eta^p}[w]_{A_{p}}w\left(\bigcup_{Q\in\mathcal{S}}Q\right).
\]
\end{lemma}

\begin{proof}First we note that
\begin{align*}
1=\frac{\mu(Q)}{\mu(Q)} & \leq\frac{1}{\eta}\frac{\mu(E_{Q})}{\mu(Q)}=\frac{1}{\eta}\frac{1}{\mu(Q)}\int_{E_{Q}}w^{\frac{1}{p}}w^{-\frac{1}{p}}\\
 & \leq\frac{1}{\eta}\left(\frac{w(E_{Q})}{\mu(Q)}\right)^{\frac{1}{p}}\left(\frac{w^{-\frac{1}{p-1}}(E_{Q})}{\mu(Q)}\right)^{\frac{1}{p'}}\\
 & \lesssim\frac{1}{\eta}\left(\frac{w(E_{Q})}{\mu(Q)}\right)^{\frac{1}{p}}\left(\frac{w^{-\frac{1}{p-1}}(1Q)}{\mu(1Q)}\right)^{\frac{1}{p'}}\\
 & \leq\frac{1}{\eta}\left(\frac{w(E_{Q})}{\mu(Q)}\right)^{\frac{1}{p}}[w]_{A_{p}}^{\frac{1}{p}}\left(\frac{\mu(1Q)}{w(1Q)}\right)^{\frac{1}{p}}
\end{align*}
and hence
\[
w(Q)\lesssim\frac{1}{\eta^{p}}[w]_{A_{p}}w(E_{Q}).
\]
Taking this into account
\begin{align*}
\sum_{Q\in\mathcal{S}}w(Q) & \lesssim\frac{1}{\eta^{p}}[w]_{A_{p}}\sum_{Q\in\mathcal{S}}w(E_{Q})=\frac{1}{\eta^{p}}[w]_{A_{p}}w\left(\bigcup_{Q\in\mathcal{S}}Q\right)
\end{align*}
and we are done. 
\end{proof}

We end this section with the following Lemma that readily follows from the definitions both in the euclidean setting and for spaces of homogeneous type.
\begin{lemma}
If $u\in A_{1}$ and $v\in A_{p}(u)$ then $uv\in A_{p}$. 
\end{lemma}

\subsection{Young functions and Orlicz averages}\label{subsec:Young}
Now, we recall that given a Young function $A:[0,\infty)\to [0,\infty)$, namely a convex, non-decreasing function such that $A(0)=0$ and $A(t)\to \infty$ when $t\to \infty$ we can define 
the average on weighted Luxemburg norm 
$$\|f\|_{A(u),Q}=\inf\left\{\lambda >0 : \frac1{u(Q)}\int_Q A\left(\frac{|f(x)|}{\lambda}\right)u(x)dx \leq 1 \right\}.$$
Also we  can define the Luxemburg norm on spaces of homogenous type just replacing the Lebesgue measure, $dx$, by the corresponding measure, $d\mu$, and cubes by balls.

It is also possible to settle a generalized H\"older inequality for Young functions. If $B^{-1}(t)C^{-1}(t)\leq c A^{-1}(t)$ for $t\geq t_0>1$ then 
$$\|fg\|_{A(u),Q}\leq c \|f\|_{B(u),Q}\|g\|_{C(u),Q}$$
We shall drop $u$ in the notation in the case of Lebesgue measure. 
For each $A$ Young function we define the associated Young function $\bar{A}$ by $\bar{A}(t)=\sup_{0\leq s<\infty} ts-A(s)$. Note that $A^{-1}(t)\bar{A}^{-1}(t)\leq 2 t$.

 A Young function $A$ is said to be submultiplicative if there exists $c_A\geq1$ such that $A(ts)\leq c_A A(t)A(s)$. 

Let $u$ be a weight and $A$ be a Young function. We define the maximal operator $M_{A(u)}^{\mathcal{F}}$ by 
$$M_{A(u)}^{\mathcal{F}}=\sup_{x\in Q \subset \mathcal{F}} \|f\|_{A(u),Q}$$
where the supremum is taken over all the cubes in the family $\mathcal{F}$. 

We said $A\in B_p$ if $$\int_{1}^{\infty} \frac{A(t)}{t^p}\frac{dt}{t}<\infty.$$ Given $1<p<\infty$, $M_A $ is bounded on $L^p$ if and only if $A\in B_p$, for more details see \cite{P95}.
Observe that if $A\in B_p$ for all $p>1$ we get $A(t)\leq c_n \kappa(\varepsilon)t^{1+\varepsilon}$ with $\kappa: (0,\infty)\to (0,\infty)$ for all $\varepsilon>0$. For example if $A(t)=t(1+\log^+t)^{\gamma}$ then $A(t)\leq (2\gamma)^{\gamma}\varepsilon^{-\gamma}t^{1+\varepsilon}$ for $t\geq e$. 

An important result that connect the average given by a Young function with the class of weights $A_p(u)$ is contained in the following Lemma.
\begin{lemma}\cite{CRR20}
Let $u$ a weight, $v\in A_p(u)$ and $\Phi$ a Young function. Then, for every cube $Q$,
$$\|f\|_{\Phi(u),Q}\leq \|f\|_{[v]_{A_p(u)}\Phi^p(uv),Q}.$$ 
\end{lemma}

Now, we recall if $b\in BMO$ then
$$\sup_Q \|b-b_Q\|_{\exp L,Q}\leq c_n \|b\|_{BMO}.$$
It is possible to define classes of symbols with ever better properties of integrability that $BMO$ symbols. Given $r>1$ we say that $b\in \text{Osc}_{\exp L^r(w)}$ if 
$$\|b\|_{ \text{Osc}_{\exp L^r(w)}}=\sup_Q \|b-b_Q\|_{\exp L^r,Q}<\infty.$$
Note that $\text{Osc}_{\exp L^r} \subsetneq BMO$ for every $r>1$. It is not hard to prove that for those classes of functions the following estimate hold
\begin{lemma}\label{OscLr} Let $j>0$, $w\in A_{\infty}$ and $b\in \text{Osc}_{\exp L^r(w)}$ with $r>1$. Then 
$$\||b-b_Q|^j\|_{\exp L^{\frac{r}{j}}(w),Q} \leq c [w]_{A_{\infty}}^{\frac{j}{r}}\|b\|_{\text{Osc}_{\exp L^r}}^j $$
\end{lemma}

To deal with the case of spaces of homogeneous type we will need the following version of the lemma above. Observe that a worse dependence on the constant appears, due to the fact that we need to change balls by cubes to use John-Nirenberg's type inequality.
\begin{lemma} Let $Q$ be a cube in a dyadic structure and $b\in Osc(L^r)$ with $r>1$. If $w\in A_{p}$ then 
\[
\|b-b_{Q}\|_{\exp L^{r}(w),Q}\lesssim\|b\|_{Osc(L^{r})}[w]_{A_{p}}^{\frac{1}{r}}[w]_{A_{\infty}}^{\frac{1}{r}}.
\]
\end{lemma}

\begin{proof}
First we note that it is not hard to check that $b\in Osc(L^r)$, implies that
\[\mu\left\{ x\in B\,:\,|f(x)-f_{B}|>t\right\} \leq e\mu(B)e^{-\left(\frac{t}{2\|f\|_{Osc(L^{r})}}\right)^{r}}.\]
Bearing that in mind we our argument observing that since
\[
|b(x)-b_{Q}|\leq|b(x)-b_{1Q}|+c\|b\|_{Osc(L^{r})}
\]
we have that 
\[
\|b-b_{Q}\|_{\exp L^{r}(w),Q}\leq\|b(x)-b_{1Q}\|_{\exp L^{r}(w),Q}+c\|b\|_{Osc(L^{r})}
\]
and hence it suffices to deal with $\|b(x)-b_{1Q}\|_{\exp L^{r}(w),Q}$.
We argue as follows
\begin{align*}
 & \frac{1}{w(Q)}\int_{Q}\left(\exp\left(\frac{|b(x)-b_{1Q}|^{r}}{\lambda^{r}}\right)-1\right)w(x)d\mu(x)\\
 & \leq\frac{1}{w(Q)}\int_{1Q}\left(\exp\left(\frac{|b(x)-b_{1Q}|^{r}}{\lambda^{r}}\right)-1\right)w(x)d\mu(x)\\
 & \leq\frac{1}{w(Q)}\int_{1Q}\exp\left(\frac{|b(x)-b_{1Q}|^{r}}{\lambda^{r}}\right)-1w(x)d\mu(x)\\
\end{align*}
Choosing $\lambda=2^{\frac{1}{r}}\|b\|_{\text{Osc}(L^{r})}c^{\frac{1}{r}}c_{X}^{\frac{1}{r}}[w]_{A_{\infty}}^{\frac{1}{r}}[w]_{A_{p}}^{\frac{1}{r}}$
we have that
\begin{align*}
 & c[w]_{A_{p}}\int_{0}^{\infty}e^{\left(1-\frac{\lambda^{r}}{\|b\|_{Osc(L^{r})}^{r}}\frac{1}{c_{X}[w]_{A_{\infty}}}\right)t}dt\\
 & \leq c[w]_{A_{p}}\int_{0}^{\infty}e^{\left(1-c[w]_{A_{p}}2\right)t}dt\\
 & =\frac{c[w]_{A_{p}}}{2c[w]_{A_{p}}-1}\leq1
\end{align*}
And this yields
\[
\|b-b_{1Q}\|_{\exp L^{r}(w),Q}\lesssim\|b\|_{\text{Osc}(L^{r})}[w]_{A_{\infty}}^{\frac{1}{r}}[w]_{A_{p}}^{\frac{1}{r}}.
\]
Gathering the estimates above we are done.
\end{proof}

\section{Proofs of main results}\label{sec:proofmain}
\subsection{Scheme of the proofs of the main results}
In this section we briefly outline the scheme that we are going to follow for each  of the proofs of the estimates in the main results. As we mentioned in the introduction, the scheme can be traced back to . Let $T$ be a linear operator and $\mathcal{M}$ maximal type and dyadic, in some sense, operator such that
$$uv\left(\left\{x\in \mathbb{R}^n: \mathcal{M}f(x)>t\right\}\right) \lesssim \int A\left(\frac{|f|}{t}\right)uv$$
where $A$ is a submultiplicative Young function. Note that by homogeneity it suffices to show that
$$uv\left(\left\{x\in \mathbb{R}^n:\frac{T(fv)(x)}{v(x)}>1\right\}\right) \lesssim \kappa_{u,v} \int A\left(|f|\right)uv$$
where $\kappa_{u,v}\geq 1$ is the constant given by the dependence on the weights involved. Taking that into account we could proceed as follows

\begin{align*}
uv\left(\left\{x\in \mathbb{R}^n:\frac{T(fv)(x)}{v(x)}>1\right\}\right)
&\leq uv\left(\left\{x\in \mathbb{R}^n:\frac{T(fv)(x)}{v(x)}>1,\mathcal{M}f(x)\leq \frac1{2}\right\}\right)
\\&+uv\left(\left\{x\in \mathbb{R}^n: \mathcal{M}f(x)>\frac1{2}\right\}\right)
\end{align*}

Since the desired estimate holds for the second term it suffices to control the first one. Let us call 
$$G=\left\{x\in \mathbb{R}^n:\frac{T(fv)(x)}{v(x)}>1,\mathcal{M}f(x)\leq \frac1{2}\right\}$$

Then it suffices to prove 
$$uv(G)\leq c_{n,T} \kappa_{u,v}\int A(|f|)uv + \frac1{2}uv(G)$$

since this yields
$$uv\left(\left\{x\in \mathbb{R}^n:\frac{T(fv)(x)}{v(x)}>1,\mathcal{M}f(x)\leq \frac1{2}\right\}\right)
\leq 2 c_{n,T} \kappa_{u,v}\int A(|f|)uv $$

\subsection{Proofs of the results in the euclidean setting}
\subsubsection{Lemmatta}
\begin{lemma} \label{RHsrho}
Let $1\leq r < \infty$. If $u\in RH_q$ with $q=2r-1$ and $u^{r}\in A_{\infty}$, then for $s=1+\frac1{2\tau_n[u^{r}]_{A_{\infty}}}$ and  any measurable subset $E\subset Q$, 
$$\frac{u^{sr}(E)}{u^{sr}(Q)}\lesssim [u]_{RH_{q}}^{\frac{q}{4}} \left(\frac{u(E)}{u(Q)}\right)^{\frac1{4}}$$
\end{lemma}

\begin{proof}
First, let see
\begin{align}\label{eq:RHq}
u^{r}(E)\leq [u]_{RH_{q}}^{q/2} \left(\frac{u(E)}{u(Q)}\right)^{1/2}u^{r}(Q)
\end{align}

Indeed, since $u\in RH_q$ we obtain
\begin{align*}
\left(\frac{u^q(Q)}{u(Q)}\right)^{1/2}
&\leq [u]_{RH_{q}}^{q/2}\left(\frac{u(Q)}{|Q|}\right)^{r -1}
\leq [u]_{RH_{q}}^{q/2} \frac{|Q|}{u(Q)}\left(\frac{u^{r}(Q)}{|Q|}\right)
= [u]_{RH_{q}}^{q/2} \frac{u^{r}(Q)}{u(Q)}.
\end{align*}

Then 
\begin{align*}
u^{r}(E)
&\leq u(E)^{1/2} u^{q}(Q)^{1/2}
\leq  u(E)^{1/2} [u]_{RH_{q}}^{q/2} \frac{u^{r}(Q)}{u(Q)^{1/2}}
\leq [u]_{RH_{q}}^{q/2} \left(\frac{u(E)}{u(Q)}\right)^{1/2}u^{r}(Q).
\end{align*}

In the other hand, since $s=1+\frac1{2\tau_n[u^{r}]_{A_{\infty}}}$ then by a similar argument as above we get 
$$\left(\frac{u^{r(2s-1)}(Q)}{u^r(Q)}\right) ^{1/2}\lesssim \frac{u^{rs}(Q)}{u^r(Q)} $$
and
\begin{equation}\label{eq:RHs}
 \frac{u^{sr}(E)}{u^{sr}(Q)}\lesssim \left(\frac{u^{r}(E)}{u^{r}(Q)}\right)^{1/2}.
 \end{equation}

Taking account \eqref{eq:RHq} and \eqref{eq:RHs}, we obtain
$$ \frac{u^{sr}(E)}{u^{sr}(Q)}\lesssim \left(\frac{u^{r}(E)}{u^{r}(Q)}\right)^{1/2}\leq  [u]_{RH_{q}}^{q/4} \left(\frac{u(E)}{u(Q)}\right)^{1/4}.$$

\end{proof}

\begin{lemma}\label{promA}
Let $1\leq r<\infty$ and $A$ be a Young function such that $A\in B_{\rho}$ for all $\rho>r$. If $u\in RH_q$ with $q=2r-1$  and $u^r\in A_{\infty}$ then for any cube $Q$ and $G$ measurable subset 
$$\|\chi_G u\|_{A,Q}\leq c_n \kappa_{u} [u]_{RH_q}^{1+\frac{q}{4r}} \langle u\rangle_{Q,1} \langle \chi_G \rangle_{Q,s}^{u}$$
with $s=4(1+\frac1{2\tau_n[u^r]_{A_{\infty}}})r$
\end{lemma}

\begin{remark}
Observe that if $A(t)=t^r(1+\log^+t)^{\gamma}$ then $A(t)\leq c_{\gamma,r} \varepsilon^{-\frac{\gamma}{r}}t^{r(1+\varepsilon)}$ and $\kappa_{u}= C [u^r]_{A_{\infty}}^{\frac{\gamma}{r}}$.
\end{remark}

\begin{proof}
Since $A\in B_{\rho}$ for all $\rho>r$ then $A(t)\leq c_n \kappa_{\varepsilon} t^{r(1+\varepsilon)}$ for all $\varepsilon >0$. Then 
\begin{equation}
\|\chi_G u\|_{A,Q} \leq c_n \kappa_{\varepsilon} \|\chi_G u\|_{r(1+\varepsilon),Q}=c_n \kappa_{\varepsilon} \|\chi_G u^r\|_{1+\varepsilon,Q}^{\frac1{r}}
\end{equation}

with  $\varepsilon =\frac1{2\tau_n[u^r]_{A_{\infty}}}$. Using Lemma  \ref{RHsrho} we have
\begin{align}
\left(\frac{u^{r(1+\varepsilon)}(G\cap Q)}{u^{r(1+\varepsilon)}(Q)}\right) ^{\frac1{1+\varepsilon}}
\leq c_n  [u]_{RH_{q}}^{\frac{q}{4(1+\varepsilon)}}
\left(\frac{u(G\cap Q)}{u(Q)}\right) ^{\frac1{4(1+\varepsilon)}}
\end{align}

In the other hand, since $1+2\varepsilon= 1+\frac1{\tau_n[u^r]_{A_{\infty}}}$ and $u\in RH_q$, with $q=2r-1>r$, we get
\begin{align}
\left(\frac{u^{r(1+\varepsilon)}(Q)}{|Q|}\right)^{\frac1{r(1+\varepsilon)}}
&\leq \left(\frac{u^r(Q)}{|Q|}\right)^{\frac1{2r(1+\varepsilon)}}\left(\frac1{|Q|}\int_Q u^{r(1+2\varepsilon)}\right)^{\frac1{2r(1+\varepsilon)}}
\nonumber
\\&\leq 2^{\frac1{r}} \left(\frac{u^r(Q)}{|Q|}\right)^{\frac1{2r(1+\varepsilon)}} \left(2\frac{u^r(Q)}{|Q|}\right)^{\frac{1+2\varepsilon}{2r(1+\varepsilon)}}=2^{\frac1{r}} \left(\frac{u^r(Q)}{|Q|}\right)^{\frac1{r}}
\nonumber
\\&\leq 2^{\frac1{r}} \left(\frac{u^q(Q)}{|Q|}\right)^{\frac1{q}} \leq c_n [u]_{RH_q} \frac{u(Q)}{|Q|}.
\end{align}

Taking account the inequalities above we obtain
\begin{align*}
\|\chi_G u\|_{A,Q} 
&\leq c_n \kappa_{u} \|\chi_G u^r\|_{1+\varepsilon,Q}^{\frac1{r}}
\\&\leq c_n \kappa_{u}  \left(\frac{u^{r(1+\varepsilon)}(Q)}{|Q|}\right)^{\frac1{r(1+\varepsilon)}}\left(\frac{u^{r(1+\varepsilon)}(G\cap Q)}{u^{r(1+\varepsilon)}(Q)}\right)^{\frac1{r(1+\varepsilon)}}
\\&\leq c_n \kappa_{u} [u]_{RH_q}^{1+\frac{q}{4r(1+\varepsilon)}}\langle u\rangle_{Q,1}  \langle\chi_G\rangle_{Q,4(1+\varepsilon)r}^{u}
\\&\leq c_n \kappa_{u} [u]_{RH_q}^{1+\frac{q}{4r}}\langle u\rangle_{Q,1}  \langle\chi_G\rangle_{Q,4(1+\varepsilon)r}^{u}.
\end{align*}
\end{proof}

Now, we recall the followings lemmas proved in \cite{CRR20}

\begin{lemma}\cite{CRR20}\label{CRRlema1}
Let $\gamma_1,\gamma_2>1$. For every $j,k$ non negative integers let 
$$\alpha_{j,k}=\min \left\{  \gamma_1 2^{-k}j^{\rho_1}, \beta \gamma_2 2^{-j}2^{-k}2^{\delta k} k^{\rho_2}  \right\}$$
where $\rho_1,\rho_2,\delta \geq 0$ and $\beta>0$. Then 
$$\sum_{j,k\geq 0} \alpha_{j,k}\leq  c_{\rho_1,\rho_2,\gamma,\delta}\gamma_1 \log_2(e+\gamma_2)^{1+\rho_1}+\frac1{2\gamma}\beta,$$
where $\gamma\geq 1$.
\end{lemma}

\begin{lemma}\cite{CRR20} \label{CRRlema2}
Let $A$ be a Young function such that $A(xy)\leq c_A A(x)A(y)$ for some $c_A\geq 1$ and $\mathcal{S}$ be a $\frac{c_A8A(2)}{1+c_A8A(2)}$-sparse family. Let $f\in C_c^{\infty}$ and $w\in A_{\infty}$ and assume that for every $Q\in \mathcal{S}$
$$2^{-j-1}\leq \|f\|_{A(w),Q}\leq 2^{-j}$$
Then for every $Q\in \mathcal{S}$ there exist $\tilde{E}_Q \subset Q$ such that
$$\sum_{Q\in \mathcal{S}}\chi_{\tilde{E}_Q}\leq c_n [w]_{A_{\infty}}$$
and 
$$w(Q)\|f\|_{A(w),Q}\leq 4c_A\frac{A(2^{j+1})}{2^{j+1}}\int_{\tilde{E}_Q}A(|f|)w.$$
\end{lemma}

\begin{lemma}\cite{CRR20}\label{CRRlema3} 
Let $w\in A_{\infty}$ and $\mathcal{S}$ be a $\eta$-sparse family of cubes. Then 
$$\sum_{Q\in \mathcal{S}} w(Q) \leq c_n [w]_{A_{\infty}}w\left(\bigcup_{Q\in \mathcal{S}} Q\right).$$
\end{lemma}

\begin{lemma}\cite{CRR20}\label{CRRlema4}
Let $A$ be a Young function such that $A(st)\leq c_A A(s)A(t)$. Let $\mathcal{D}_j$, $j=1,\dots,k$ be dyadic grids and let $w$ a weight. Then 
$$w\left(\left\{ x\in \mathbb{R}^n : M_{A(w)}^{\mathcal{F}} f(x) >t\right\}\right) \leq c_A c_n \int_{\mathbb{R}^n}A\left(\frac{|f(x)|}{t}\right)w(x)dx$$
where $\displaystyle \mathcal{F}=\bigcup_{j=1}^{k}\mathcal{D}_j$. 
\end{lemma}

We end this lemmatta section with the following key Lemma.
\begin{lemma}\label{cuentaSjk}
Let $1< p < \infty$, $\xi\geq 1$ and $\rho \geq 0$. Let $u\in A_1$, $v\in A_p(u)$ and $A(t)=t\log(e+t)^{\rho}$ and let $\mathcal{S}$ be a sparse family. Then, if $f\in L_c^{\infty}$ and  $g=\chi_G$ where  $G\subset \{x: M_{A(uv)} f (x)\leq \frac1{2}\}$ is a set of finite measure, we have that for every $\tau_{u,v}, \gamma\geq 1$ 
\begin{align*}
\tau_{u,v}&\sum_{Q\in \mathcal{S}} \|f\|_{A(uv),Q}\|g\|_{L^{\xi},Q}  uv(Q)
 \\&
\leq  c_{n,p} \tau_{u,v}[uv]_{A_{\infty}}\log(e+\tau_{u,v}[uv]_{A_{\infty}}[v]_{A_p(u)})^{1+\rho} \int_{\mathbb{R}^n} A(|f|)uv + \frac1{2 \gamma} uv(G).
\end{align*}
\end{lemma}

\begin{proof}

Taking into account that $G$, is a subset of the set where  $M_{A(uv)}(f)\leq \frac1{2}$ we can split $\mathcal{S}$ as follows $Q\in \mathcal{S}_{j,k}$, $j,k\geq 0$ if  
\begin{align*}
2^{-j-1}<&\|f\|_{A(uv),Q} \leq 2^{-j}
\\
2^{-k-1}<&\|g\|_{L^{\xi}(u),Q}\leq 2^{-k}
\end{align*}

Let us define 
$$ s_{j,k} = \sum_{Q\in \mathcal{S}_{j,k}} \|f\|_{A(uv),Q}  \|g\|_{L^{\xi}(u),Q} uv(Q) $$

We claim that
\begin{align*}
s_{j,k} \leq \begin{cases}
c_n 2^{-k}[uv]_{A_{\infty}}j^{\rho}\int_{\mathbb{R}} A(|f|)uv
\\
c_{n,p} [uv]_{A_{\infty}}[v]_{A_p(u)} 2^{-j}2^{k(\xi p-1)} uv(G)
\end{cases}
\end{align*}

For the top estimate we use Lemma \ref{CRRlema2} with $w=uv$ and $A(t)=t(1+\log^+t)^{\rho}$, and we have
$$uv(Q)\|f\|_{A(uv),Q}\leq c j^{\rho}\int_{\tilde{E}_Q} A(|f|)uv$$
with
$$\sum_{Q\in \mathcal{S}_{j,k}} \chi_{\tilde{E}_Q}(x)\leq \lceil c_n [uv]_{A_{\infty}} \rceil.$$

Then, 
\begin{align*}
s_{j,k}&\leq c 2^{-k}j^{\rho}\sum_{Q\in \mathcal{S}_{j,k}} \int_{\tilde{E}_Q} A(|f|)uv
\leq c_n [uv]_{A_{\infty}}2^{-k}j^{\rho}\int_{\mathbb{R}^n}A(|f|)uv.
\end{align*}

For the lower estimate, using Lemma \ref{CRRlema3}
\begin{align*}
s_{j,k}
&\leq 2^{-j}2^{-k} \sum_{Q\in \mathcal{S}_{j,k}} uv(Q)  
\\&\leq c_n [uv]_{A_{\infty}} 2^{-j}2^{-k} uv\left(\cup_{Q\in \mathcal{S}_{j,k}}Q\right)
\\&\leq c_n [uv]_{A_{\infty}} 2^{-j}2^{-k} uv\left(\left\{x\in \mathbb{R}^n : M_u(g)^{\frac1{\xi}}>2^{-k-1}\right\}\right)
\end{align*}
Since $v\in A_p(u)$ we have 
$$\frac1{u(Q)}\int_Q gu \leq \left( \frac{[v]_{A_p(u)}}{uv(Q)}\int_Q guv\right)^{\frac1{p}}$$

Then, 

\begin{align*}
s_{j,k}
&\leq c_n [uv]_{A_{\infty}} 2^{-j}2^{-k} uv\left(\left\{x\in \mathbb{R}^n : \left([v]_{A_p(u)}M_{uv}(g)\right)^{\frac1{\xi p}}>2^{-k-1}\right\}\right)
\\&\leq c_n [uv]_{A_{\infty}} 2^{-j}2^{-k} uv\left(\left\{x\in \mathbb{R}^n : M_{uv}(g)>2^{-\xi p(k+1)}[v]_{A_p(u)}^{-1}\right\}\right)
\\&\leq c_{n,p} [uv]_{A_{\infty}}[v]_{A_p(u)} 2^{-j}2^{k(\xi p-1)}uv(G)  
\end{align*}

Combining the estimates above 
\begin{align*}
&\tau_{u,v}\sum_{Q\in \mathcal{S}} \|f\|_{A(uv),Q}\|g\|_{L^{\xi},Q}  uv(Q)
=
\tau_{u,v} \sum_{k=0}^{\infty}\sum_{j=0}^{\infty} s_{j,k}
\\&\leq \sum_{j,k=0}^{\infty}\min\big\{ c_{n}  \tau_{u,v}[uv]_{A_{\infty}} 2^{-k}j^{\rho}\int_{\mathbb{R}^n} A(|f|)uv,
c_{n,p} \tau_{u,v}[uv]_{A_{\infty}} [v]_{A_p(u)} 2^{-j}2^{k(\xi p-1)} uv(G)\big\}
\end{align*}
 We end the proof using Lemma \ref{CRRlema1} with $\gamma_1 = c_n\tau_{u,v}[uv]_{A_{\infty}}\int_{\mathbb{R}^n} A(|f|)uv$, 
 \\ $\gamma_2=c_{n,p}\tau_{u,v}[uv]_{A_{\infty}}[v]_{A_p(u)}$, $\beta=uv(G)$, $\delta=\xi p$, $\gamma=3^n$, $\rho_1=\rho$ and $\rho_2=0$.

\end{proof}

\subsubsection{Proof of Theorem \ref{MixtaT}}
Let $G=\{x : \frac{T(fv)(x)}{v(x)}> 1\}\setminus \{x: M_{uv} f (x)>\frac1{2}\}$ and assume that $\|f\|_{L^1(uv)}=1$. If we denote $g=\chi_G$, for the sparse domination we have
\begin{align*}
uv(G)&\leq\left|\int T(fv)gu dx\right|
\lesssim \sum_{j=1}^{3^n}\sum_{Q\in \mathcal{S}_j} \left(\int_Q fv\right) \|gu\|_{A,Q}  
\end{align*}
Since $u\in A_1\cap RH_r$ then $u^r\in A_1 \subset A_{\infty}$, then we can take $s=4(1+\frac1{2\tau_n[u^r]_{A_{\infty}}})r$.
By Lemma \ref{promA}, 
$$\|g u\|_{A,Q}\leq c_n \kappa_{u} [u]_{RH_q}^{1+\frac{q}{4r}} \langle u\rangle_{Q,1} \langle g \rangle_{Q,s}^{u}.$$

Since $u\in A_1$, 
\begin{align*}
uv(G)
&\lesssim \sum_{j=1}^{3^n}\sum_{Q\in \mathcal{S}_j} \left(\int_Q fv\right) \|gu\|_{A,Q}  
\\&  \leq c_n \kappa_{u} [u]_{RH_q}^{1+\frac{q}{4r}} [u]_{A_1}\sum_{j=1}^{3^n}\sum_{Q\in \mathcal{S}_j} \langle f \rangle_{Q,1}^{uv}  uv(Q)  \langle g \rangle_{Q,s}^{u} 
\end{align*}

 Now, we apply Lemma \ref{cuentaSjk}, with $\xi=s$, $\rho=0$, $\tau_{u,v}=c_n \kappa_{u} [u]_{RH_q}^{1+\frac{q}{4r}} [u]_{A_1}$ and $\gamma=3^n$, then 
\begin{align*}
uv(G)
&  \leq c_n \kappa_{u} [u]_{RH_q}^{1+\frac{q}{4r}} [u]_{A_1}\sum_{j=1}^{3^n}\sum_{Q\in \mathcal{S}_j} \langle f \rangle_{Q,1}^{uv}  uv(Q)  \langle g \rangle_{Q,s}^{u} 
\\& \leq c_{n,p} \kappa_{u} [u]_{RH_q}^{1+\frac{q}{4r}} [u]_{A_1}[uv]_{A_{\infty}}\log(e+\kappa_{u} [u]_{RH_q}^{1+\frac{q}{4r}} [u]_{A_1}[uv]_{A_{\infty}}[v]_{A_p(u)})+ \frac1{2} uv(G)
\end{align*}

\subsubsection{Proof of Theorems \ref{multiSimb} and \ref{conmT} }
We provide the proof of Theorem \ref{multiSimb} first and at the end we show how to adjust the argument to settle Theorem \ref{conmT}.
Let $G=\{x : \frac{T_{\bf b}(fv)(x)}{v(x)}> 1\}\setminus \{x: M_{\varphi_{\frac1{r}}(uv)} f (x)>\frac1{2}\}$, with $\varphi_{\frac1{r}}(t)=t\log(e+t)^{\frac1{r}}$ and $g=\chi_G$, for the sparse domination we have
\begin{align*}
uv(G)&\leq\left|\int T_{\bf b}(fv)gu \right|
\lesssim \sum_{j=1}^{3^n}\sum_{h=0}^m \sum_{\sigma \in C_h(b)} \sum_{Q\in \mathcal{S}_j}
  \left(\int_Q fv|b-b_Q|_{\sigma '}\right)  \|gu|b-b_Q|_{\sigma}\|_{A,Q}  
\end{align*}

Let $\xi=4(1+\frac1{2\tau_n[u^s]_{A_\infty}})s$.
By  $u\in A_1$,   Theorem \ref{promA} and  H\"older inequality we obtain
\begin{align*}
&  \sum_{Q\in \mathcal{S}}
  \left(\int_Q fv|b-b_Q|_{\sigma '}\right)  \|gu|b-b_Q|_{\sigma}\|_{A,Q}  
 \\&  \leq   \sum_{Q\in \mathcal{S}}
  \left(\int_Q fv|b-b_Q|_{\sigma '}\right)  \|gu\|_{B,Q} \prod_{i\in \sigma}\|b-b_Q\|_{\exp L^{r_i},Q}  
\\&\lesssim \kappa_u[u]_{RH_{q}}^{1+\frac{q}{4s}}\sum_{Q\in \mathcal{S}} \left(\int_Q fv|b-b_Q|_{\sigma '}\right)  \|g\|_{L^{\xi}(u),Q}\frac{u(Q)}{|Q|}\prod_{i\in \sigma}\|b-b_Q\|_{\exp L^{r_i},Q}  
\\&\leq \kappa_u[u]_{RH_{q}}^{1+\frac{q}{4s}}[u]_{A_1}\sum_{Q\in \mathcal{S}} \|f|b-b_Q|_{\sigma '}\|_{uv,Q}  \|g\|_{L^{\xi}(u),Q}\prod_{i\in \sigma}\|b-b_Q\|_{\exp L^{r_i},Q}uv(Q)
\\&\lesssim \kappa_u[u]_{RH_{q}}^{1+\frac{q}{4s}}[u]_{A_1}[uv]_{A_{\infty}}^{\frac{m-h}{r}}\sum_{Q\in \mathcal{S}} \prod_{i\in 1}^m \|b-b_Q\|_{\exp L^{r_i},Q}   \|f\|_{L\log L^{\frac1{r}}(uv),Q}  \|g\|_{L^{\xi}(u),Q} uv(Q)
\\&= \kappa_u[u]_{RH_{q}}^{1+\frac{q}{4s}}[u]_{A_1}[uv]_{A_{\infty}}^{\frac{m-h}{r}} \|{\bf b}\| \sum_{Q\in \mathcal{S}} \|f\|_{L\log L^{\frac1{r}}(uv),Q}  \|g\|_{L^{\xi}(u),Q} uv(Q)
\end{align*}

Now, we apply Lemma \ref{cuentaSjk} with $\rho=\frac1{r}$, $A(t)=\varphi_{\frac1{r}}(t)$, $\tau_{u,v}= \kappa_u[u]_{RH_{q}}^{1+\frac{q}{4s}}[u]_{A_1}[uv]_{A_{\infty}}^{\frac{m-h}{r}}$ and  $\gamma=3^n {{m}\choose{h}}$ (recall that ${m}\choose{h}$ is the cardinality of $C_h(b)$), then 
\begin{align*}
 \sum_{\sigma \in C_h(b)}\sum_{Q\in \mathcal{S}}&
  \left(\int_Q fv|b-b_Q|_{\sigma '}\right)  \|gu|b-b_Q|_{\sigma}\|_{A,Q}  
\\&  \leq c_{n,T} C_{u,v} \int_{\R} \varphi_{\frac{1}{r}}\left(\frac{|f|\| {\bf b}\|}{\lambda}\right) uv + \frac1{3^n2}uv(G)
\end{align*}
where
\begin{align*}
C_{uv} =  \kappa_u [u]_{RH_{q}}^{1+\frac{q}{4s}} [u]_{A_1} [uv]_{A_{\infty}}^{1+\frac{m-h}{r}}\log(e+[uv]_{A_{\infty}}^{1+\frac{m-h}{r}} \kappa_u[u]_{RH_{q}}^{1+\frac{q}{4s}}[u]_{A_1}[v]_{A_p(u)})^{1+\frac1{r}}\qed 
\end{align*}

To settle Theorem \ref{conmT}, due to the following Lemma, that follows from ideas in \cite{CUMT21}, it suffices to apply the argument above just to the cases in which $\sigma=\emptyset$ or in which $\sigma$ contains the $m$ ``copies'' of $b$.
\begin{lemma}\label{lemaBMO}Given $b\in L^m_\text{loc}$, a sparse family $\mathcal{S}$, a positive integer $m$ and $h\in\{0,\dots,m\}$ we have that
\begin{align*}
\mathcal{A}_{A,\mathcal{S}}^{m-h}(b,f)(x)
\leq  \sum_{Q\in \mathcal{S}}
  |b-b_Q|^m \|f\|_{A,Q} \chi_Q(x) +  \sum_{Q\in \mathcal{S}}  \|f|b-b_Q|^m\|_{A,Q} \chi_Q(x). 
\end{align*}
\end{lemma}

\begin{proof}
\begin{align*}
\mathcal{A}_{A,\mathcal{S}}^{m-h}(b,f)(x)
&=\sum_{Q\in \mathcal{S}}   |b(x)-b_Q|^{m-h}  \|f|b-b_Q|^h\|_{A,Q}\chi_Q(x)
\\&=\sum_{Q\in \mathcal{S}}     \|f|b(x)-b_Q|^{m-h}|b-b_Q|^h\|_{A,Q}\chi_Q(x)
\\&\leq \sum_{Q\in \mathcal{S}}     \|f\max\{|b(x)-b_Q|,|b-b_Q|\}^m\|_{A,Q}\chi_Q(x)
\\&\leq  \sum_{Q\in \mathcal{S}}
  |b(x)-b_Q|^m \|f\|_{A,Q} \chi_Q(x) + \sum_{Q\in \mathcal{S}}  \|f|b-b_Q|^m\|_{A,Q} \chi_Q(x) 
  \end{align*}
\end{proof}

\subsection{Proofs of the results in spaces of homogeneous type}
\subsubsection{Lemmata}

Our first Lemma is the following.
\begin{lemma}
\label{RHsHom} Let $1\leq r<\infty$. If $u\in RH_{q}$ with $q=2r-1$
and $u^{r}\in A_{\infty}$, then for $s=1+\frac{1}{2\tau_{n}[u^{r}]_{A_{\infty}}}$
and any measurable subset $E\subset Q$, 
\[
\frac{u^{sr}(E)}{u^{sr}(c_{d}Q)}\lesssim[u]_{RH_{q}}^{q/4}\left(\frac{u(E)}{u(c_{d}Q)}\right)^{\frac{1}{4}}
\]
\end{lemma}

\begin{proof}
First, let us see 
\begin{align}
u^{r}(E)\leq[u]_{RH_{q}}^{q/2}\left(\frac{u(E)}{u(1Q)}\right)^{1/2}u^{r}(c_{d}Q)\label{eq:RHqHom}
\end{align}

Indeed, since $u\in RH_{q}$ we obtain 
\begin{align*}
\left(\frac{u^{q}(1Q)}{u(c_{d}Q)}\right)^{1/2} & =\left(\frac{\frac{u^{q}(1Q)}{\mu(c_{d}Q)}}{\frac{u(c_{d}Q)}{\mu(c_{d}Q)}}\right)^{1/2}\leq[u]_{RH_{q}}^{q/2}\left(\frac{\left(\frac{u(c_{d}Q)}{\mu(c_{d}Q)}\right)^{q}}{\frac{u(c_{d}Q)}{\mu(c_{d}Q)}}\right)^{1/2}=[u]_{RH_{q}}^{q/2}\left(\frac{u(c_{d}Q)}{\mu(c_{d}Q)}\right)^{r-1}\\
 & =[u]_{RH_{q}}^{q/2}\frac{\mu(c_{d}Q)}{u(c_{d}Q)}\left(\frac{u(c_{d}Q)}{\mu(c_{d}Q)}\right)^{r}\leq[u]_{RH_{q}}^{q/2}\frac{\mu(c_{d}Q)}{u(c_{d}Q)}\frac{u^{r}(c_{d}Q)}{\mu(c_{d}Q)}\\
 & =[u]_{RH_{q}}^{q/2}\frac{u^{r}(c_{d}Q)}{u(c_{d}Q)}
\end{align*}

Taking that into account, 
\begin{align*}
u^{r}(E) & \leq u(E)^{1/2}u^{q}(1Q)^{1/2}=u(E)^{1/2}u(c_{d}Q)^{\frac{1}{2}}\left(\frac{u^{q}(1Q)}{u(c_{d}Q)}\right)^{1/2}\\
 & \leq u(E)^{1/2}[u]_{RH_{q}}^{q/2}\frac{u^{r}(c_{d}Q)}{u(c_{d}Q)^{1/2}}=[u]_{RH_{q}}^{q/2}\left(\frac{u(E)}{u(c_{d}Q)}\right)^{1/2}u^{r}(c_{d}Q).
\end{align*}

In the other hand, since $s=1+\frac{1}{2\tau_{n}[u^{r}]_{A_{\infty}}}$
we have that 
\begin{align*}
\left(\frac{u^{r(2s-1)}(1Q)}{u^{r}(c_{d}Q)}\right)^{\frac{1}{2}} & =\left(\frac{u^{r(2s-1)}(1Q)}{\mu(1Q)}\frac{\mu(1Q)}{u^{r}(c_{d}Q)}\right)^{\frac{1}{2}}\lesssim\left(\frac{u^{r}(c_{d}Q)}{\mu(c_{d}Q)}\right)^{\frac{2s-1}{2}}\left(\frac{\mu(1Q)}{u^{r}(c_{d}Q)}\right)^{\frac{1}{2}}\\
 & \leq\left(\frac{u^{r}(c_{d}Q)}{\mu(c_{d}Q)}\right)^{s}\left(\frac{\mu(c_{d}Q)}{u^{r}(c_{d}Q)}\right)\leq\left(\frac{u^{rs}(c_{d}Q)}{\mu(c_{d}Q)}\right)\left(\frac{\mu(c_{d}Q)}{u^{r}(c_{d}Q)}\right)\\
 & =\frac{u^{rs}(c_{d}Q)}{u^{r}(c_{d}Q)}
\end{align*}
Summarizing 
\[
\left(\frac{u^{r(2s-1)}(1Q)}{u^{r}(c_{d}Q)}\right)^{\frac{1}{2}}\lesssim\frac{u^{rs}(c_{d}Q)}{u^{r}(c_{d}Q)}.
\]
Observe that relying upon this estimate 
\begin{align*}
\frac{u^{sr}(E)}{u^{sr}(c_{d}Q)} & =\frac{u^{sr-\frac{r}{2}+\frac{r}{2}}(E)}{u^{sr}(c_{d}Q)}\leq\frac{u^{2r(s-1)}(E)^{\frac{1}{2}}u^{r}(E)^{\frac{1}{2}}}{u^{sr}(c_{d}Q)}\\
 & \leq\left(\frac{u^{2r(s-1)}(1Q)}{u^{r}(c_{d}Q)}\right)^{\frac{1}{2}}\frac{u^{r}(c_{d}Q)^{\frac{1}{2}}u^{r}(E)^{\frac{1}{2}}}{u^{sr}(c_{d}Q)}\\
 & \lesssim\frac{u^{rs}(c_{d}Q)}{u^{r}(c_{d}Q)}\frac{u^{r}(c_{d}Q)^{\frac{1}{2}}u^{r}(E)^{\frac{1}{2}}}{u^{sr}(c_{d}Q)}=\left(\frac{u^{r}(E)}{u^{r}(c_{d}Q)}\right)^{1/2}
\end{align*}
Hence, 
\begin{equation}
\frac{u^{sr}(E)}{u^{sr}(c_{d}Q)}\lesssim\left(\frac{u^{r}(E)}{u^{r}(c_{d}Q)}\right)^{1/2}.\label{eq:RHsHom}
\end{equation}

Taking account (\ref{eq:RHqHom}) and (\ref{eq:RHsHom}), we obtain
\begin{align*}
\frac{u^{sr}(E)}{u^{sr}(c_{d}Q)} & \lesssim\left(\frac{u^{r}(E)}{u^{r}(c_{d}Q)}\right)^{1/2}\lesssim\left(\frac{[u]_{RH_{q}}^{q/2}\left(\frac{u(E)}{u(c_{d}Q)}\right)^{1/2}u^{r}(c_{d}Q)}{u^{r}(c_{d}Q)}\right)^{1/2}\\
 & =[u]_{RH_{q}}^{q/4}\left(\frac{u(E)}{u(c_{d}Q)}\right)^{\frac{1}{4}}.
\end{align*}
\end{proof}
We continue with the following Lemma which is a counterpart of Lemma \ref{promA}.
\begin{lemma}
\label{promAHom} Let $1\leq r<\infty$ and $A$ be a Young function
such that $A\in B_{\rho}$ for all $\rho>r$. If $u\in RH_{q}$ with
$q=2r-1$ and $u^{r}\in A_{\infty}$ then for any cube $Q$ and $G$
measurable subset 
\[
\|\chi_{G}u\|_{A,1Q}\leq c_{X}\kappa_{u}[u]_{RH_{q}}^{1+\frac{q}{4r}}\langle u\rangle_{c_{d}^{3}Q,1}\langle\chi_{G}\rangle_{c_{d}Q,s}^{u}
\]
with $s=4(1+\frac{1}{2\tau_{n}[u^{r}]_{A_{\infty}}})r$. 
\end{lemma}

\begin{proof}
Since $A\in B_{\rho}$ for all $\rho>r$ then $A(t)\leq c\kappa_{\varepsilon}t^{r(1+\varepsilon)}$
for all $\varepsilon>0$. Then 
\begin{equation}
\|\chi_{G}u\|_{A,1Q}\leq c\kappa_{\varepsilon}\|\chi_{G}u\|_{r(1+\varepsilon),1Q}=c\kappa_{\varepsilon}\|\chi_{G}u^{r}\|_{1+\varepsilon,1Q}^{\frac{1}{r}}
\end{equation}
with $\varepsilon=\frac{1}{2\tau_{n}[u^{r}]_{A_{\infty}}}$. Using
Lemma \ref{RHsHom} we have 
\begin{align}
\left(\frac{u^{r(1+\varepsilon)}(G\cap Q)}{u^{r(1+\varepsilon)}(c_{d}Q)}\right)^{\frac{1}{1+\varepsilon}}\leq c_{X}[u]_{RH_{q}}^{\frac{q}{4(1+\varepsilon)}}\left(\frac{u(G\cap Q)}{u(c_{d}Q)}\right)^{\frac{1}{4(1+\varepsilon)}}
\end{align}
In the other hand, since $1+2\varepsilon=1+\frac{1}{\tau_{n}[u^{r}]_{A_{\infty}}}$
and $u\in RH_{q}$, with $q=2r-1>r$, we get 
\begin{align}
\left(\frac{u^{r(1+\varepsilon)}(c_{d}Q)}{\mu(c_{d}Q)}\right)^{\frac{1}{r(1+\varepsilon)}} & =\left(\frac{u^{\frac{r}{2}+\frac{r}{2}+\varepsilon}(c_{d}Q)}{\mu(c_{d}Q)}\right)^{\frac{1}{r(1+\varepsilon)}}\nonumber \\
 & \leq\left(\frac{u^{r}(c_{d}Q)}{\mu(c_{d}Q)}\right)^{\frac{1}{2r(1+\varepsilon)}}\left(\frac{1}{\mu(c_{d}Q)}\int_{c_{d}Q}u^{r(1+2\varepsilon)}\right)^{\frac{1}{2r(1+\varepsilon)}}\\
 & \leq2^{\frac{1}{r}}\left(\frac{u^{r}(1Q)}{\mu(c_{d}Q)}\right)^{\frac{1}{2r(1+\varepsilon)}}\left(c\frac{u^{r}(c_{d}^{2}Q)}{\mu(c_{d}^{2}Q)}\right)^{\frac{1+2\varepsilon}{2r(1+\varepsilon)}}\lesssim\left(\frac{u^{r}(c_{d}^{2}Q)}{\mu(c_{d}^{2}Q)}\right)^{\frac{1}{r}}\nonumber \\
 & \leq\left(\frac{u^{q}(c_{d}^{2}Q)}{\mu(c_{d}^{2}Q)}\right)^{\frac{1}{q}}\lesssim[u]_{RH_{q}}\frac{u(c_{d}^{3}Q)}{\mu(c_{d}^{3}Q)}
\end{align}
Taking account the inequalities above we obtain 
\begin{align*}
\|\chi_{G}u\|_{A,1Q} & \leq c\kappa_{u}\|\chi_{G}u^{r}\|_{1+\varepsilon,1Q}^{\frac{1}{r}}\\
 & \lesssim\kappa_{u}\left(\frac{\mu(c_{d}Q)}{\mu(1Q)}\right)^{\frac{1}{r(1+\varepsilon)}}\left(\frac{u^{r(1+\varepsilon)}(c_{d}Q)}{\mu(c_{d}Q)}\right)^{\frac{1}{r(1+\varepsilon)}}\left(\frac{u^{r(1+\varepsilon)}(G\cap1Q)}{u^{r(1+\varepsilon)}(c_{d}Q)}\right)^{\frac{1}{r(1+\varepsilon)}}\\
 & \leq c_{X}\kappa_{u}[u]_{RH_{q}}^{1+\frac{q}{4r(1+\varepsilon)}}\langle u\rangle_{c_{d}^{3}Q,1}\langle\chi_{G}\rangle_{c_{d}Q,4(1+\varepsilon)r}^{u}\\
 & \leq c_{X}\kappa_{u}[u]_{RH_{q}}^{1+\frac{q}{4r}}\langle u\rangle_{c_{d}^{3}Q,1}\langle\chi_{G}\rangle_{c_{d}Q,4(1+\varepsilon)r}^{u}.
\end{align*}
\end{proof}
\begin{lemma} \label{CRRlema2Hom}Let $w\in A_{p}$ and let $A$
be a submultiplicative Young function and $\mathcal{S}$ a dyadic
$\frac{4c_{A}A(4)c[w]_{A_{p}}}{1+4c_{A}A(4)c[w]_{A_{p}}}$-sparse
family. Let $f\in\mathcal{C}_{c}^{\infty}$ and assume that for every
$Q\in\mathcal{S}$ 
\[
2^{-j-1}\leq\langle f\rangle_{A(L)(w)Q}\leq2^{-j}.
\]
Then for every $Q\in\mathcal{S}$ there exists $\tilde{E}_{Q}\subseteq Q$
such that 
\[
\sum_{Q\in\mathcal{S}}\chi_{\tilde{E_{Q}}}(x)\leq c_{X}[w]_{A_{\infty}}
\]
and 
\[
w(Q)\|f\|_{A(w),Q}\leq4\frac{A(2^{j+2})}{2^{j+2}}\int_{\tilde{E_{Q}}}A\left(|f|\right)w.
\]
\end{lemma}

\begin{proof} We split the family $\mathcal{S}$ in the following
way 
\[
\begin{split}\mathcal{S}^{0} & =\{\text{Maximal cubes in }\mathcal{S}\}\\
\mathcal{S}^{1} & =\{\text{Maximal cubes in }\mathcal{S}\setminus\mathcal{S}^{0}\}\\
 & \dots\\
\mathcal{S}^{i} & =\{\text{Maximal cubes in }\mathcal{S}\setminus\cup_{r=0}^{i-1}\mathcal{S}^{r}\}
\end{split}
\]
Recall that since $w\in A_{\infty}$ we have that, for each cube $Q$
and each measurable subset $E\subset Q$, 
\[
w(E)\leq c\left(\frac{\mu(E)}{\mu(1Q)}\right)^{\frac{1}{c_{X}[w]_{A_{\infty}}}}w(c_{d}Q)
\]
Now observe that if $Q\in\mathcal{S}^{i}$ and $J_{1}=\bigcup_{P\in\mathcal{S}^{i+1},\,P\subsetneq Q}P$,
if we call $\kappa=4c_{A}A(4)c[w]_{A_{p}}$, 
\begin{align*}
\mu(J_{1}) & =\mu\left(\bigcup_{P\in\mathcal{S}^{i+1},\,P\subsetneq Q}P\right)\leq\bigcup_{P\in\mathcal{S}^{i+1},\,P\subsetneq Q}\mu(P)\\
 & \leq\left(\frac{1+\kappa}{\kappa}-1\right)\mu(Q)=\frac{1}{\kappa}\mu(Q)
\end{align*}
Arguing by induction, if we denote $J_{\nu}=\bigcup_{P\in\mathcal{S}^{i+\nu},\,P\subset Q}P$
then we have that 
\[
\mu(J_{\nu})\leq\left(\frac{1}{\kappa}\right)^{\nu}\mu(Q)\leq\left(\frac{1}{\kappa}\right)^{\nu}\mu(1Q)
\]
and hence, 
\[
w(J_{\nu})\leq2\left(\frac{1}{4c_{A}A(4)c[w]_{A_{p}}}\right)^{\frac{\nu}{c_{X}[w]_{A_{\infty}}}}w(c_{d}Q).
\]
In particular if we choose $\nu=\left\lfloor  c_{X}[w]_{A_{\infty}}\right\rfloor $,
then by Lemma \ref{CRRlema3Hom} 
\[
w(J_{\nu})\leq\frac{2}{\kappa}w(c_{d}Q)\leq\frac{c[w]_{A_{p}}}{2c_{A}A(4)c[w]_{A_{p}}}w(Q)=\frac{1}{2c_{A}A(4)}w(Q).
\]
Let $Q\in\mathcal{S}^{i}$. and let $\tilde{E_{Q}}=Q\setminus\bigcup_{P\in\mathcal{S}^{i+\left\lceil c_{X}[w]_{A_{\infty}}\right\rceil }}P$.
Then we have that 
\[
\begin{split} & w(Q)\|f\|_{A(w),Q}\\
 & \leq w(Q)\left\{ 2^{-j-2}+\frac{2^{-j-2}}{w(Q)}\int_{Q}A\left(2^{j+2}|f|\right)w\right\} \\
 & \leq w(Q)2^{-j-2}+\frac{1}{2^{j+2}}\int_{Q}A\left(2^{j+2}|f|\right)w\\
 & \leq w(Q)2^{-j-1}+\frac{1}{2^{j+2}}\int_{\tilde{E_{Q}}}A\left(2^{j+2}|f|\right)w+\frac{1}{2^{j+2}}\sum_{P\in\mathcal{S}_{j,k}^{i+\left\lceil c_{X}[w]_{A_{\infty}}\right\rceil }}\int_{P}A\left(2^{j+2}|f|\right)w\\
 & \leq w(Q)2^{-j-2}+\frac{A(2^{j+2})}{2^{j+2}}\int_{\tilde{E_{Q}}}A\left(|f|\right)w+\frac{1}{2^{j+2}}\sum_{P\in\mathcal{S}_{j,k}^{i+\left\lceil c_{X}[w]_{A_{\infty}}\right\rceil }}\int_{P}A\left(2^{j+2}|f|\right)w.
\end{split}
\]
Observe that we can bound the last term as follows 
\[
\begin{split} & \sum_{P\in\mathcal{S}_{j,k}^{i+\left\lceil c_{X}[w]_{A_{\infty}}\right\rceil }}\int_{P}A\left(2^{j+2}|f|\right)w\leq c_{A}A(4)\sum_{P\in\mathcal{S}_{j,k}^{i+\left\lceil c_{X}[w]_{A_{\infty}}\right\rceil }}w(P)\frac{1}{w(P)}\int_{P}A\left(2^{j}|f|\right)w\\
 & \leq A(4)\sum_{P\in\mathcal{S}_{j,k}^{i+\left\lceil c_{X}[w]_{A_{\infty}}\right\rceil }}w(P)=c_{A}A(4)w\left(J_{\left\lceil c_{X}[w]_{A_{\infty}}\right\rceil }\right)\\
 & \leq\frac{c_{A}A(4)}{2}\frac{1}{2c_{A}A(4)}w(Q)\leq\frac{1}{4}w(Q).
\end{split}
\]
Hence 
\[
\begin{split}w(Q)\|f\|_{A(w),Q} & \leq\frac{1}{2^{j+2}}w(Q)+\frac{A(2^{j+2})}{2^{j+2}}\int_{\tilde{E_{Q}}}A\left(|f|\right)w+\frac{1}{2^{j+2}}\frac{1}{4}w(Q)\\
 & \leq\left(\frac{1}{2}+\frac{1}{4}\right)w(Q)\|f\|_{A(w),Q}+\frac{A(2^{j+2})}{2^{j+2}}\int_{\tilde{E_{Q}}}A\left(|f|\right)w\\
 & =\frac{3}{4}w(Q)\|f\|_{A(w),Q}+\frac{A(2^{j+2})}{2^{j+2}}\int_{\tilde{E_{Q}}}A\left(|f|\right)w,
\end{split}
\]
from which the desired conclusion readily follows. \end{proof}%
We end this lemmatta with the following key Lemma.
\begin{lemma}
\label{lem:Reabsorcioneth}Let $1<p<\infty$, $s\geq 1$ and $\rho\ge0$. Let
$u\in A_{1}$, $v\in A_{p}(u)$ and $A(t)=t\log^{\rho}(e+t)$. Assume
that $\mathcal{S}$ is a $\frac{4c_{A}A(4)c[w]_{A_{p}}}{1+4c_{A}A(4)c[w]_{A_{p}}}$-sparse
family. Then if $f\in L_{c}^{\infty}$ and $g=\chi_{G}$ where $G\subset\left\{ x\in X\,:\,M_{A(uv)}f\leq\frac{1}{2}\right\} $
is a set of finite measure, we have that for every $\tau_{u,v},\gamma\geq1$
\begin{align*}
 & \tau_{u,v}\sum_{Q\in\mathcal{S}}\|f\|_{A(uv),Q}\langle g\rangle_{c_{d}Q,s}^{u}uv(Q)\\
 & \leq c_{\gamma}\tau_{u,v}[uv]_{A_{\infty}}\log(e+\tau_{u,v}[uv]_{A_{p}}^{2}[v]_{A_{p}(u)})^{1+\rho}\int_{X}A(|f|)d\mu+\frac{1}{2\gamma}uv(G).
\end{align*}
\end{lemma}
\begin{proof}
 Assume that $\|f\|_{A(uv)}=1$. We split the sparse family $\mathcal{S}$
as follows. We say that $Q\in\mathcal{S}_{j,k}$, $j,k\geq0$ if 
\begin{align*}
2^{-j-1} & <\|f\|_{A(uv),Q}\leq2^{-j}\\
2^{-k-1} & <\langle g\rangle_{c_{d}Q,s}^{u}\leq2^{-k}
\end{align*}
Let 
\[
s_{j,k}=\sum_{Q\in\mathcal{S}_{j,k}}\|f\|_{A(uv),Q}\langle g\rangle_{c_{d}Q,s}^{u}uv(Q)
\]
We claim that 
\begin{align*}
s_{j,k}\leq\begin{cases}
c_{n}[uv]_{A_{\infty}}2^{-k}j^{\rho}\\
c_{n,p}[uv]_{A_{p}}^{2}[v]_{A_{p}(u)}2^{-j}2^{k(sp-1)}uv(G)
\end{cases}
\end{align*}
For the top estimate we argue as follows. Using Lemma \ref{CRRlema2Hom}
we have that there exists a set $\tilde{E}_{Q}\subset Q$ such that
\[
\sum_{Q\in\mathcal{S}_{j,k}}\chi_{\tilde{E}_{Q}}(x)\leq\lceil c_{n}[uv]_{A_{\infty}}\rceil
\]
and 
\[
\|f\|_{A(uv),Q}uv(Q)\lesssim\frac{A(2^{j+2})}{2^{j+2}}\int_{\tilde{E}_{Q}}A(|f|)uv\simeq j^{\rho}\int_{\tilde{E}_{Q}}A(|f|)uv.
\]
Then 
\begin{align*}
s_{j,k} & \leq2^{-k}\sum_{Q\in\mathcal{S}_{j,k}}\|f\|_{A(uv),Q}uv(Q)\leq2^{-k}\sum_{Q\in\mathcal{S}_{j,k}}j^{\rho}\int_{\tilde{E}_{Q}}A(|f|)uv\\
 & \leq c_{n}[uv]_{A_{\infty}}2^{-k}j^{\rho}\int_{X}A(|f|)uv=[uv]_{A_{\infty}}2^{-k}j^{\rho}
\end{align*}
For the lower estimate, by Lemma \ref{CRRlema3Hom} 
\begin{align*}
s_{j,k} & \leq2^{-j}2^{-k}\sum_{Q\in\mathcal{S}_{j,k}}uv(Q)\\
 & \leq c_{n}[uv]_{A_{p}}2^{-j}2^{-k}uv\left(\cup_{Q\in\mathcal{S}_{j,k}}Q\right)\\
 & \leq c_{n}[uv]_{A_{p}}2^{-j}2^{-k}uv\left(\left\{ x\in\mathbb{R}^{n}:M_{u}(g)^{\frac{1}{s}}>2^{-k-1}\right\} \right)
\end{align*}
Since $v\in A_{p}(u)$ we have 
\[
\frac{1}{u(B)}\int_{B}gu\leq\left(\frac{[v]_{A_{p}(u)}}{uv(B)}\int_{B}guv\right)^{\frac{1}{p}}
\]
Then, 
\begin{align*}
s_{j,k} & \leq c_{n}[uv]_{A_{p}}2^{-j}2^{-k}uv\left(\left\{ x\in\mathbb{R}^{n}:\left([v]_{A_{p}(u)}M_{uv}(g)\right)^{\frac{1}{sp}}>2^{-k-1}\right\} \right)\\
 & \leq c_{n}[uv]_{A_{p}}2^{-j}2^{-k}uv\left(\left\{ x\in\mathbb{R}^{n}:M_{uv}(g)>2^{-sp(k+1)}[v]_{A_{p}(u)}^{-1}\right\} \right)\\
\text{} & \leq c_{n,p}[uv]_{A_{p}}^{2}[v]_{A_{p}(u)}2^{-j}2^{k(sp-1)}uv(G)
\end{align*}
Combining the estimates above 
\begin{align*}
 & \tau_{u,v}\sum_{j,k=0}^{\infty}s_{j,k}\\
 & \leq\sum_{j,k=0}^{\infty}\min\left\{ \tau_{u,v}[uv]_{A_{\infty}}2^{-k}j^{\rho},c_{n,p}\tau_{u,v}[uv]_{A_{p}}^{2}[v]_{A_{p}(u)}2^{-j}2^{k(sp-1)}uv(G)\right\} 
\end{align*}
where $K_{u,v}=\kappa_{u}[u]_{RH_{q}}^{1+\frac{q}{4r}}[u]_{A_{1}}[uv]_{A_{p}}$.
We end the proof using Lemma \ref{CRRlema1}, with $\gamma_{1}=\tau_{u,v}[uv]_{A_{\infty}}$,
$\gamma_{2}=c_{n,p}\tau_{u,v}[uv]_{A_{p}}^{2}[v]_{A_{p}(u)}$, $\beta=uv(G)$,
$\delta=sp$, $\rho_{1}=\rho$, $\rho_{2}=0$.
\end{proof}

\subsubsection{Proof of Theorem \ref{thm:Teth}}
Let $G=\{x:\frac{T(fv)(x)}{v(x)}>1\}\setminus\{x:M_{uv}f(x)>\frac{1}{2}\}$
and assume that $\|f\|_{L^{1}(uv)}=1$. If we denote $g=\chi_{G}$,
by \eqref{bilinearsparseTETH}
\begin{align*}
uv(G) & \leq\left|\int T(fv)gudx\right|\lesssim\frac{1}{1-\varepsilon}\sum_{j=1}^{l}\sum_{Q\in\mathcal{S}_{j}}\left(\int_{Q}fv\right)\|gu\|_{A,Q}
\end{align*}
We shall choose $\varepsilon=\frac{16c[uv]_{A_{p}}}{1+16c[uv]_{A_{p}}}$.
Hence 
\[
uv(G)\lesssim[uv]_{A_{p}}\sum_{j=1}^{l}\sum_{Q\in\mathcal{S}_{j}}\left(\int_{Q}fv\right)\|gu\|_{A,Q}.
\]
Since $u\in A_{1}\cap RH_{r}$, then $u^{r}\in A_{1}\subset A_{\infty}$,
and taking $s=4(1+\frac{1}{2\tau_{n}[u^{r}]_{A_{\infty}}})r$, by
Lemma \ref{promAHom}, 
\[
\|gu\|_{A,Q}\leq c_{d}\kappa_{u}[u]_{RH_{q}}^{1+\frac{q}{4r}}\langle u\rangle_{c_{d}^{3}Q,1}\langle g\rangle_{c_{d}Q,s}^{u}.
\]
Taking the estimates above into account and bearing in mind that $u\in A_{1}$,
\begin{align*}
uv(G) & \leq c[uv]_{A_{p}}\sum_{j=1}^{l}\sum_{Q\in\mathcal{S}_{j}}\left(\int_{Q}fv\right)\|gu\|_{A,Q}\\
 & \leq c\kappa_{u}[uv]_{A_{p}}[u]_{RH_{q}}^{1+\frac{q}{4r}}\sum_{j=1}^{l}\sum_{Q\in\mathcal{S}_{j}}\left(\int_{Q}fv\right)\langle u\rangle_{c_{d}^{3}Q,1}\langle g\rangle_{c_{d}Q,s}^{u}\\
 & \leq c\kappa_{u}[uv]_{A_{p}}[u]_{RH_{q}}^{1+\frac{q}{4r}}[u]_{A_{1}}\sum_{j=1}^{l}\sum_{Q\in\mathcal{S}_{j}}\langle f\rangle_{Q,1}^{uv}\langle g\rangle_{c_{d}Q,s}^{u}uv(Q).
\end{align*}
A direct application of Lemma \ref{lem:Reabsorcioneth} with $\tau_{u,v}=c\kappa_{u}[uv]_{A_{p}}[u]_{RH_{q}}^{1+\frac{q}{4r}}[u]_{A_{1}}$
ends the proof.

\subsubsection{Proof of Theorem \ref{thm:bTeth}}
We shall assume that $\|b\|_{Osc(L^{r_{i}})}=1$ for every $i$. Let $G=\{x:\frac{T_{{\bf b}}(fv)(x)}{v(x)}>1\}\setminus\{x:M_{\varphi_{\frac{1}{r}}(uv)}f(x)>\frac{1}{2}\}$,
with $\varphi_{\frac{1}{r}}(t)=t\log(e+t)^{\frac{1}{r}}$ and $g=\chi_{G}$,
by \eqref{bilinearsparseCommETH}, with $\varepsilon=\frac{4c_{\varphi_{\frac{1}{r}}}\varphi_{\frac{1}{r}}(4)c[uv]_{A_{p}}}{1+4c_{\varphi_{\frac{1}{r}}}\varphi_{\frac{1}{r}}(4)c[uv]_{A_{p}}}$
\begin{align*}
uv(G) & \leq\left|\int T_{{\bf b}}(fv)gu\right|\lesssim\frac{1}{1-\varepsilon}\sum_{s=1}^{l}\sum_{h=0}^{m}\sum_{\sigma\in C_{h}(b)}\sum_{Q\in\mathcal{S}_{s}}\left(\int_{Q}fv|b-b_{Q}|_{\sigma'}d\mu\right)\|gu|b-b_{Q}|_{\sigma}\|_{A,Q}\\
 & \lesssim[uv]_{A_{p}}\sum_{s=1}^{l}\sum_{h=0}^{m}\sum_{\sigma\in C_{h}(b)}\sum_{Q\in\mathcal{S}_{s}}\left(\int_{Q}fv|b-b_{Q}|_{\sigma'}d\mu\right)\|gu|b-b_{Q}|_{\sigma}\|_{A,Q}
\end{align*}
Let $\xi=4(1+\frac{1}{2\tau_{n}[u^{s}]_{A_{\infty}}})s$. By $u\in A_{1}$,
Theorem \ref{promAHom} and H\"older inequality we obtain 
\begin{align*}
 & \sum_{Q\in\mathcal{S}}\left(\int_{Q}fv|b-b_{Q}|_{\sigma'}d\mu\right)\|gu|b-b_{Q}|_{\sigma}\|_{A,Q}\\
 & \leq\sum_{Q\in\mathcal{S}}\left(\int_{Q}fv|b-b_{Q}|_{\sigma'}d\mu\right)\|gu\|_{B,Q}\prod_{i\in\sigma}\|b-b_{Q}\|_{\exp L^{r_{i}},Q}\\
 & \leq\kappa_{u}[u]_{RH_{q}}^{1+\frac{q}{4s}}\sum_{Q\in\mathcal{S}}\left(\int_{Q}fv|b-b_{Q}|_{\sigma'}d\mu\right)\|g\|_{L^{\xi}(u),c_{d}Q}\frac{u(Q)}{|Q|}\prod_{i\in\sigma}\|b-b_{Q}\|_{\exp L^{r_{i}},Q}\\
 & \leq\kappa_{u}[u]_{RH_{q}}^{1+\frac{q}{4s}}[u]_{A_{1}}\sum_{Q\in\mathcal{S}}\|f|b-b_{Q}|_{\sigma'}\|_{uv,Q}\|g\|_{L^{\xi}(u),c_{d}Q}\prod_{i\in\sigma}\|b-b_{Q}\|_{\exp L^{r_{i}},Q}uv(Q)\\
 & \leq\kappa_{u}[u]_{RH_{q}}^{1+\frac{q}{4s}}[u]_{A_{1}}\left([uv]_{A_{p}}[uv]_{A_{\infty}}\right)^{\frac{m-h}{r}}\sum_{Q\in\mathcal{S}}\prod_{i\in1}^{m}\|b-b_{Q}\|_{\exp L^{r_{i}},Q}\|f\|_{L\log L^{\frac{1}{r}}(uv),Q}\|g\|_{L^{\xi}(u),c_{d}Q}uv(Q)\\
 & =\kappa_{u}[u]_{RH_{q}}^{1+\frac{q}{4s}}[u]_{A_{1}}\left([uv]_{A_{p}}[uv]_{A_{\infty}}\right)^{\frac{m-h}{r}}\|{\bf b}\|\sum_{Q\in\mathcal{S}}\|f\|_{L\log L^{\frac{1}{r}}(uv),Q}\|g\|_{L^{\xi}(u),c_{d}Q}uv(Q)
\end{align*}
Taking into account the definition of $G$, since we removed the set
where $M_{u,v}(f)>\frac{1}{2}$ we can split $\mathcal{S}$ as follows.
We say that $Q\in\mathcal{S}_{j,k}$, $j,k\geq0$ if 
\begin{align*}
2^{-j-1}< & \|f\|_{L\log L^{\frac{1}{r}}(uv),Q}\leq2^{-j}\\
2^{-k-1}< & \|g\|_{L^{\xi}(u),c_{d}Q}\leq2^{-k}.
\end{align*}
Let us define 
\[
s_{j,k}=\sum_{Q\in\mathcal{S}_{j,k}}\|f\|_{L\log L^{\frac{1}{r}}(uv),Q}\|g\|_{L^{\xi}(u),c_{d}Q}uv(Q)
\]
We claim that 
\begin{align*}
s_{j,k}\leq\begin{cases}
c_{n}2^{-k}[uv]_{A_{\infty}}j^{\frac{1}{r}}\int_{\mathbb{R}}\varphi_{\frac{1}{r}}(|f|)uv\\
c_{n,p}[uv]_{A_{p}}^{2}[v]_{A_{p}(u)}2^{-j}2^{k(\xi p-1)}uv(G)
\end{cases}
\end{align*}
For the lower estimate we argue as we did in the Theorem \ref{MixtaT}.
For the top estimate we use Lemma \ref{CRRlema2} with $w=uv$ and
$A(t)=\Phi_{\frac{1}{r}}(t)=t(1+\log^{+}t)^{\frac{1}{r}}$, and we
have 
\[
uv(Q)\|f\|_{L\log L^{\frac{1}{r}}(uv),Q}\leq cj^{\frac{1}{r}}\int_{\tilde{E}_{Q}}\Phi_{\frac{1}{r}}(|f|)uv
\]
with 
\[
\sum_{Q\in\mathcal{S}_{j,k}}\chi_{\tilde{E}_{Q}}(x)\leq\lceil c_{n}[uv]_{A_{\infty}}\rceil.
\]
Then, 
\begin{align*}
s_{j,k} & \leq c2^{-k}j^{\frac{1}{r}}\sum_{Q\in\mathcal{S}_{j,k}}\int_{\tilde{E}_{Q}}\Phi_{\frac{1}{r}}(|f|)uv\leq c_{n}[uv]_{A_{\infty}}2^{-k}j^{\frac{1}{r}}\int_{\mathbb{R}}\Phi_{\frac{1}{r}}(|f|)uv.
\end{align*}
Combining the estimates above 
\begin{align*}
 & uv(G)\leq l c\sum_{s=1}^{l}\sum_{h=0}^{m}\sum_{\sigma\in C_{h}(b)}\kappa_{u}[u]_{RH_{q}}^{1+\frac{q}{4s}}[u]_{A_{1}}\left([uv]_{A_{p}}[uv]_{A_{\infty}}\right)^{\frac{m-h}{r}}[uv]_{A_{p}}\sum_{j,k=0}^{\infty}s_{j,k}\\
 & \leq \Gamma_{l,b}\sum_{h=0}^{m}\sum_{j,k=0}^{\infty}\min\big\{ c_{n}\tau_{u,v,h}[uv]_{A_{\infty}}2^{-k}j^{\frac{1}{r}}\int_{\mathbb{R}}\Phi_{\frac{1}{r}}(|f|)uv,c_{n,p}\tau_{u,v,h}[uv]_{A_{p}}^{2}[v]_{A_{p}(u)}2^{-j}2^{k(\xi p-1)}uv(G)\}
\end{align*}
where $\tau_{u,v,h}=\kappa_{u}[u]_{RH_{q}}^{1+\frac{q}{4s}}[u]_{A_{1}}\left([uv]_{A_{p}}[uv]_{A_{\infty}}\right)^{\frac{m-h}{r}}[uv]_{A_{p}}$ and
$\Gamma_{l,b}=l\sum_{h=0}^{m}\sum_{\sigma\in C_{h}(b)}$.
We end the proof using Lemma \ref{CRRlema1} with $\gamma_{1}=c_{n}\tau_{u,v,h}[uv]_{A_{\infty}}\int_{\mathbb{R}}\Phi_{\frac{1}{r}}(|f|)uv$,
$\gamma_{2}=c_{n,p}\tau_{u,v,h}[uv]_{A_{p}}^{2}[v]_{A_{p}(u)}$, $\beta=uv(G)$,
$\delta=\xi p$, $\gamma=\Gamma_{l,b}$, $\rho_{1}=\frac{1}{r}$ and $\rho_{2}=0$. 

\section{Sparse domination and applications of the main results}\label{sec:applications}
Note that for each operator for which a bilinear sparse bounds as the ones presented in the first section of this work hold, the corresponding endpoint estimates in the main results hold as well. 
We recall that given a Young function $A$, we say that $T$ is an $A$-Hörmander operator if 
\[\|T\|_{L^{2}\rightarrow L^{2}}<\infty\] 
and $T$ admits the following representation 
\[Tf(x)=\int_{X}K(x,y)f(y)d\mu(y)\] 
with $K$ belonging to the class $\mathcal{H}_{A}$, namely satisfying that 
\[H_{K,A}=\max\left\{ H_{K,A,1},H_{K,A,2}\right\} <\infty\] where
\[\begin{split}H_{K,A,1}=\sup_{B}\sup_{x,z\in\frac{1}{2}B}\sum_{k=1}^{\infty}\mu(2^kB)\left\Vert \left(K(x,\cdot)-K(z,\cdot)\right)\chi_{2^{k}B\setminus2^{k-1}B}\right\Vert _{A,2^{k}B}<\infty\\
H_{K,A,2}=\sup_{B}\sup_{x,z\in\frac{1}{2}B}\sum_{k=1}^{\infty}\mu(2^kB)\left\Vert \left(K(\cdot,x)-K(\cdot,z)\right)\chi_{2^{k}B\setminus2^{k-1}B}\right\Vert _{A,2^{k}B}<\infty.
\end{split}
\]
Observe that the class of $L^\infty$-H\"ormander operators contains the class of Calder\'on-Zygmund operators. 
A number of applications of $A$-H\"ormander classes of operators are contained in \cite{LMPR,LMRT}. Among them it is worth mentioning differential transform operators which are $\exp(L^{\frac1{1+\varepsilon}})$-H\"ormander operators with $\varepsilon>0$, and multipliers that are $L^r\log(L^r)$-H\"ormander operators with $r>1$.

The following result gathers pointwise sparse domination results for $\bar{A}$-Hörmander
operators in the euclidean setting.

\begin{teo}
Let A be a submultiplicative Young function, let $T$ be a $\bar{A}$-Hörmander operator
and let $f\in C_{c}^{\infty}$. Let $\varepsilon\in(0,1)$. Then there
exist $3^{n}$ dyadic lattices $\mathcal{D}_{j}$ and $\varepsilon$-sparse
families $\mathcal{S}_{j}\subset\mathcal{D}_{j}$ such that 
\begin{itemize}
\item \cite{LORR,IFRR20}
\[
|T_{b}^{m}f(x)|\leq C_{n,m,T}\frac{1}{1-\varepsilon}\sum_{j=1}^{3^n}\mathcal{A}_{A,\mathcal{S}_{j}}f(x)
\]
where
\[
\mathcal{A}_{A,\mathcal{S}}(f)(x)=\sum_{Q\in\mathcal{S}}\|f\|_{A,Q}\chi_{Q}(x)
\]
\item \cite{IFRR20} if $m$ is a non-negative integer and $b\in L_{\text{loc}}^{m}(\mathbb{R}^{n})$,
then 
\[
|T_{b}^{m}f(x)|\leq C_{n,m,T}\frac{1}{1-\varepsilon}\sum_{j=1}^{3^n}\sum_{h=0}^{m}{m \choose h}\mathcal{A}_{A,\mathcal{S}_{j}}^{m,h}(b,f)(x)
\]
where 
\[
\mathcal{A}_{A,\mathcal{S}}^{m,h}(b,f)(x)=\sum_{Q\in\mathcal{S}}|b(x)-b_{Q}|^{m-h}\|f|b-b_{Q}|^{h}\|_{A,Q}\chi_{Q}(x).
\]
\item \cite{RR18} for $b_{1},\dots,b_{m}\in L_{\text{loc}}^{1}(\mathbb{R}^{n})$ such
that $\||b|_{\sigma}\|_{A,Q}<\infty$ for every cube $Q$ and for
every $\sigma\in C_{j}(b)$ where $j\in\{1,\dots,m\}$, then 
\begin{equation}
|T_{\bf b}f(x)|\leq c_{n,m,T}\frac{1}{1-\varepsilon}\sum_{j=1}^{3^n}\sum_{h=0}^{m}\sum_{\sigma\in C_{h}(b)}\mathcal{A}_{A,\mathcal{S}_{j}}^{\sigma}(b,f)(x)\label{eq:SparseSigma}
\end{equation}
where 
\[
\mathcal{A}_{A,\mathcal{S}}^{\sigma}(b,f)(x)=\sum_{Q\in\mathcal{S}}|b(x)-b_{Q}|_{\sigma'}\|f|b-b_{Q}|_{\sigma}\|_{A,Q}\chi_{Q}(x).
\]
\end{itemize}
\end{teo}
Observe that the required bilinear estimates \eqref{bilinearsparseT}, \eqref{bilinearsparseTb} and \eqref{bilinearsparsecomm}  for the main results hold since if $G$ is any of the operators above, then
\[\left|\int_{\mathbb{R}^n}Gfg\right|=\left|\int_{\mathbb{R}^n}fG^*g\right|\]
and the same sparse bounds for $G$ hold as well for $G^*$ and then  \eqref{bilinearsparseT}, \eqref{bilinearsparseTb} and \eqref{bilinearsparsecomm} readily follow. Hence the main results allow us to derive the corresponding estimates for the aforementioned operators.
\begin{remark}Note that in the case of $\bar{A}$-Hörmander operators the estimates above
appeared first in \cite{IFRR20} under some additional technical assumption
on $A$. However such an assumption can be dropped, as it was shown in \cite{LO} for $T$. In the case of commutators \eqref{eq:SparseSigma}
appeared first in \cite{RR18} under some additional technical condition
on $A$. Later on in \cite[Theorem 3.5]{HLO}, some particular cases of that result were recovered. It is possible as well, to provide the quantitative version in terms of the sparseness constant following ideas in \cite{HytNotes}. In the next subsection we will provide such a result in the realm of spaces of homogeneous type.
\end{remark}

\begin{remark}
In view of the aforementioned sparse domination results, it is worth noting that results obtained here further extend results in \cite{CRR20} and also allow to recover, for instance, results for $T_{\bf b}$, with $T$ being a Calder\'on-Zygmund operator, recently obtained in \cite{BCP21}. 
\end{remark}

\subsection{A sparse domination result for $T_{\bf b}$ commutators in spaces of homogeneous type and applications}
In this subsection we provide a full argument, extending \eqref{eq:SparseSigma}
to spaces of homogeneous type, which contains ${\bf b}=\emptyset$
as a particular case, and allows us to drop the aforementioned technical
condition in \cite{IFRR20}. Note that this result also extends the commutator bound from \cite{DGKLW}. We remit the reader to the latter mentioned paper and to \cite{L21} and the references therein for some further insight on sparse domination on spaces of homogeneous type. We will also show how to apply our main results to $A$-H\"ormander operators and their commutators.

We begin recalling that given $C$ a submultiplicative Young function
\[
\mu\left(\left\{ x\in X\,:\,M_{C}f(x)>t\right\} \right)\leq\|M_{C}\|\int_{X}C\left(\frac{|f|}{t}\right)d\mu
\]
The proof of this fact can be obtained relying upon covering by dyadic structures and then using the same argument as in, for instance, \cite[Lemma 2.6]{LORR}.

We recall as well that given $\alpha>0$ we can define the operator $\mathcal{M}_{T,\alpha}^{\#}$, that was introduced in \cite{LO} by
\[
\mathcal{M}_{T,\alpha}^{\#}f(x)=\sup_{B\ni x}\underset{y,z\in B}{\esssup}|T(f\chi_{X\setminus\alpha B})(y)-T(f\chi_{X\setminus\alpha B})(z)|.
\]
The statement of the sparse domination result we intend to prove is the following. 
\begin{teo}
\label{thm:ThmSparseComm}Let $(X,d,\mu)$ be a space of homogeneous
type and $\mathcal{D}$ a dyadic system with parameters $c_{0}$,
$C_{0}$ and $\delta$. Let us fix $\alpha\geq\frac{3c_{d}^{2}}{\delta}$
and let $f:X\rightarrow\mathbb{R}$ be a boundedly supported function
such that $f\in L^{\infty}(X)$. Assume as well that $b_{1},\dots,b_{m}\in L^{\infty}.$
Let $\text{\ensuremath{A}}$ and $B$ be submultiplicative Young functions and $C=\max(A,B)$,
and assume that there exist non-increasing functions $\psi$ and $\phi$
such that for every $Q\in\mathcal{D}$ and any supported function
$g\in L^{\infty}(X)$ 
\[
\mu\left(\{x\in Q:|T(g\chi_{Q})(x)|>\psi(\rho)\|g\|_{A,Q}\}\right)\le\rho\mu(Q)\quad(0<\rho<1)
\]
and 
\[
\mu\left(\{x\in Q:\mathcal{M}_{T,\alpha}^{\#}(g\chi_{Q})(x)>\phi(\rho)\|g\|_{B,Q}\}\right)\le\rho\mu(Q)\quad(0<\rho<1)
\]
Then, given $\varepsilon\in(0,1)$, there exists a $(1-\varepsilon)$-sparse
family $\mathcal{S}\subset\mathcal{D}$ such that 
\[
|T_{{\bf b}}f(x)|\lesssim\kappa_{\xi,\varepsilon,C}\sum_{h=0}^{m}\sum_{\sigma\in C_{h}({\bf b})}\sum_{Q\in\mathcal{S}}\||b-b_{\alpha Q}|_{\sigma}f\|_{C,\alpha Q}|b-b_{\alpha Q}|_{\sigma'}\chi_{Q}(x)
\]
where 
\[
\kappa_{\xi,\varepsilon,C}=\left(2\xi\left(\frac{\varepsilon}{3c_{1}c_{2}}\right)+\xi\left(\frac{1}{3c_{1}c_{2}}\right)\frac{\varepsilon}{3c_{1}c_{2}}\|M_{C}\|\right)
\]
and $\xi(\rho)=\psi(\rho)+\phi(\rho)$ with $c_{1}$ and $c_{2}$
being constants depending on the parameters defining $\mathcal{D}$.
Furthermore, there exist $0<c_{0}\leq C_{0}<\infty$, $0<\delta<1$,
$\gamma\geq1$ and $k\in\mathbb{N}$ such that there are dyadic systems
$\mathcal{D}_{1},\dots,\mathcal{D}_{k}$ with parameters $c_{0},C_{0}$
and $\delta$ and $k$ $(1-\varepsilon)$-sparse families $\mathcal{S}_{i}\subset\mathcal{D}_{i}$
such that 
\[
|T_{{\bf b}}f(x)|\lesssim\kappa_{\xi,\rho,C}\sum_{h=0}^{m}\sum_{\sigma\in C_{h}({\bf b})}\sum_{Q\in\mathcal{S}}\sum_{j=1}^{k}\sum_{Q\in\mathcal{S}_{j}}\||b-b_{Q}|_{\sigma}f\|_{C,Q}|b-b_{Q}|_{\sigma'}\chi_{Q}(x).
\]
\end{teo}
Before settling this result let us show how to provide applications from it. Observe that if an operator $G$ satisfies the bound
\[\mu\left(\{x\in Q:|Gf(x)|>\lambda\}\right)\le{C_G}\int_{X}A\left(\frac{|f|}{\lambda}\right)\]
then we have that if $\varphi(\rho)\geq1$ then
\[\begin{split}
\mu\left(\{x\in Q:|G(g\chi_Q)(x)|>\varphi(\rho)\|g\|_{A,Q}\}\right)&\le{C_G}\int_{X}A\left(\frac{|g\chi_Q|}{\varphi(\rho)\|g\|_{A,Q}}\right)\\ &\le{C_G}\frac{1}{\varphi(\rho)} \int_{Q}A\left(\frac{|g|}{\|g\|_{A,Q}}\right)\\
&\le{C_G}\frac{1}{\varphi(\rho)}\mu(Q). 
\end{split}\]
Hence we have that choosing, $\varphi(t)=\frac{C_G}{t}$, since $\rho\in(0,1)$, then
\[\mu\left(\{x\in Q:|G(g\chi_Q)(x)|>\varphi(\rho)\|f\|_{A,Q}\}\right)\le\rho\mu(Q). \]
Observe that this would be a suitable choice for $\varphi$ in order to apply Theorem \ref{thm:ThmSparseComm}. Furthermore, note that for this choice of $\varphi$, then
\[\begin{split}
2\varphi\left(\frac{\varepsilon}{3c_{1}c_{2}}\right)+\varphi\left(\frac{1}{3c_{1}}\right)\frac{\varepsilon}{3c_{1}c_{2}}\|M_{C}\|&=2\frac{C_G}{\frac{\varepsilon}{3c_{1}c_{2}}}+\frac{C_G}{\frac{1}{3c_{1}}}\frac{\varepsilon}{3c_{1}c_{2}}\|M_{C}\|\\
&\leq4\frac{c_1c_2C_G\|M_C\|}{\varepsilon}
\end{split}
\]
Taking these ideas into account we have the following corollary

\begin{cor}
Let $(X,d,\mu)$ be a space of homogeneous
type and $\mathcal{D}$ a dyadic system with parameters $c_{0}$,
$C_{0}$ and $\delta$. Let us fix $\alpha\geq\frac{3c_{d}^{2}}{\delta}$
and let $f:X\rightarrow\mathbb{R}$ be a boundedly supported function
such that $f\in L^{\infty}(X)$. Assume as well that $b_{1},\dots,b_{m}\in L^{\infty}.$
Let $\text{\ensuremath{A}}$ and $B$ be Young functions and $C=\max(A,B)$,
and assume that there exist non-increasing functions $\psi$ and $\phi$
such that for every $Q\in\mathcal{D}$ and any boundedly supported function
$g\in L^{\infty}(X)$ 
\begin{equation}\label{eq:weak}
\mu\left(\{x\in Q:|Tg(x)|>\lambda\}\right)\le\|T\|\int_{X}A\left(\frac{|f|}{\lambda}\right)
\end{equation}
and 
\begin{equation}\label{eq:sharp}
\mu\left(\{x\in Q:\mathcal{M}_{T,\alpha}^{\#}(g)>\lambda\}\right)\le\|\mathcal{M}_{T,\alpha}^{\#}\|\int_{X}B\left(\frac{|f|}{\lambda}\right)
\end{equation}
Then, given $\varepsilon\in(0,1)$, there exists a $\varepsilon$-sparse
family $\mathcal{S}\subset\mathcal{D}$ such that 
\[
|T_{{\bf b}}f(x)|\lesssim\kappa\sum_{h=0}^{m}\sum_{\sigma\in C_{h}({\bf b})}\sum_{Q\in\mathcal{S}}\||b-b_{\alpha Q}|_{\sigma}f\|_{C,\alpha Q}|b-b_{\alpha Q}|_{\sigma'}\chi_{Q}(x)
\]
where 
\[
\kappa=\frac{c_1c_2\left(\|T\|+\|\mathcal{M}_{T,\alpha}^{\#}\|\right)\|M_C\|}{1-\varepsilon}
\]
and $\xi(\rho)=\psi(\rho)+\phi(\rho)$ with $c_{1}$ and $c_{2}$
being constants depending on the parameters defining $\mathcal{D}$.
Furthermore, there exist $0<c_{0}\leq C_{0}<\infty$, $0<\delta<1$,
$\gamma\geq1$ and $k\in\mathbb{N}$ such that there are dyadic systems
$\mathcal{D}_{1},\dots,\mathcal{D}_{k}$ with parameters $c_{0},C_{0}$
and $\delta$ and $k$ $(1-\varepsilon)$-sparse families $\mathcal{S}_{i}\subset\mathcal{D}_{i}$
such that 
\[
|T_{{\bf b}}f(x)|\lesssim\kappa\sum_{h=0}^{m}\sum_{\sigma\in C_{h}({\bf b})}\sum_{Q\in\mathcal{S}}\sum_{j=1}^{k}\sum_{Q\in\mathcal{S}_{j}}\||b-b_{Q}|_{\sigma}f\|_{C,Q}|b-b_{Q}|_{\sigma'}\chi_{Q}(x).
\]
\end{cor}

If $T$ is an $\bar{\Psi}$-H\"ormander operator, then we have as well that it is not hard to check that 
\eqref{eq:weak} holds for $A(t):=t$ and \eqref{eq:sharp} holds for the Young function $B(t):=\Psi(t)$. Hence, the sparse control in the corollary above holds. Note that since
\[\left|\int_XT_{{\bf b}}f(x)g(x)d\mu(x)\right|=\left|\int_XT^*_{{\bf b}}g(x)f(x)d\mu(x)\right|\] 
and $T^*$ is as well an $\bar{\Psi}$-H\"ormander operator, then
\[\left|\int_XT_{{\bf b}}f(x)g(x)d\mu(x)\right|\lesssim\frac{1}{1-\varepsilon}\sum_{h=0}^{m}\sum_{\sigma\in C_{h}({\bf b})}\sum_{j=1}^{k}\sum_{Q\in\mathcal{S}_{j}}\||b-b_{Q}|_{\sigma}g\|_{\Psi,Q}\int_{Q}|f||b-b_{Q}|_{\sigma'}.\]
This yields that if $T$ is an $\bar{\Psi}$-H\"ormander operator, then $T$ itself and its commutators $T_{{\bf b}}$ satisfy the sparse domination conditions \eqref{bilinearsparseTETH} and \eqref{bilinearsparseCommETH} respectively, and hence the estimates obtained in Theorems \ref{thm:Teth}  and  \ref{thm:bTeth} hold as well for $T$ and $T_{{\bf b}}$.

\begin{remark}
The results that we have just mentioned in the preceding lines extend some results from \cite{CRR20} and the  estimate for $T_{\bf b}$, with $T$ being a Calder\'on-Zygmund operator, recently obtained in \cite{BCP21} to spaces of homogeneous type.
\end{remark}

\subsubsection{Proof of Theorem \ref{thm:ThmSparseComm}}
To settle Theorem \ref{thm:ThmSparseComm} we need a few Lemmas. 
The first of them is an adaption of ideas in \cite{LN} and was settled in \cite{DGKLW} and relates sparse and Carleson families.
\begin{lemma}
Let $\mathcal{D}$ be a dyadic system. If $\mathcal{S}\subset\mathcal{D}$ is a $\Lambda$-Carleson with $\Lambda>1$, then it is a $\frac{1}{\Lambda}$-sparse family. Conversely, if $\mathcal{S}$ is a $\varepsilon$-sparse family with $\varepsilon\in(0,1)$ then $\mathcal{S}$ is a $\frac{1}{\varepsilon}$-Carleson family.
\end{lemma}

The second of them is contained in \cite[Section 6.3]{LN}. Although there it is stated in the euclidean setting the same proof works as well for dyadic systems in spaces of homogeneous type. 
\begin{lemma}\label{LemmaCarleson}
Let $\mathcal{D}$ be a dyadic system. If $\mathcal{S}\subset\mathcal{D}$
is a $\Lambda$-Carleson family and $t\geq2$ then $\mathcal{S}=\cup_{i=1}^{t}\mathcal{S}_{i}$
each $\mathcal{S}_{i}$ is a $1+\frac{\Lambda-1}{t}$-Carleson family. 
\end{lemma}

The third and last lemma we need is the following one, which contains the key estimate to perform the iterative process that will allow us to settle Theorem \ref{thm:ThmSparseComm}. 

\begin{lemma}
\label{lem:It-1}Let $(X,d,\mu)$ be a space of homogeneous type and
$\mathcal{D}$ a dyadic system with parameters $c_{0}$, $C_{0}$
and $\delta$. Let us fix $\alpha\geq\frac{3c_{d}^{2}}{\delta}$ and
let $f$ be a boundedly supported
function such that $f\in L^\infty(X)$. Let $\text{\ensuremath{A}}$
and $B$ be Young functions and $C=\max(A,B)$, and assume that there
exist non-increasing functions $\psi$ and $\phi$ such that for every
$Q\in\mathcal{D}$\textup{, and every }boundedly supported function
$g\in L^{\infty}(X)$ 
\[
\mu\left(\{x\in Q:|T(g\chi_{Q})(x)|>\psi(\rho)\|g\|_{A,Q}\}\right)\le\rho\mu(Q)\quad(0<\rho<1)
\]
and 
\[
\mu\left(\{x\in Q:\mathcal{M}_{T,\alpha}^{\#}(g\chi_{Q})(x)>\phi(\rho)\|g\|_{B,Q}\}\right)\le\rho\mu(Q)\quad(0<\rho<1)
\]
Then, given $\varepsilon\in(0,1)$ there exist disjoint subcubes $Q_{j}\in\mathcal{D}(Q)$
such that 
\[
\sum_{j}\mu(Q_{j})\leq\varepsilon\mu(Q)
\]
and for every $\sigma\in C_{h}(\bf{b})$ and $h=0,\dots,m$, 
\begin{align*}
 & \left|T_{{\bf b}}(f\chi_{\alpha Q})(x)\chi_{Q}-\sum_{j}T_{{\bf b}}(f\chi_{\alpha P_{j}})(x)\chi_{P_{j}}(x)\right|\\
 & \leq\kappa_{\xi,\varepsilon,C}\sum_{h=0}^{m}\sum_{\sigma\in C_{h}({\bf b})}|b(x)-c_{Q}|_{\sigma}\|f|b-c_{Q}|_{\sigma'}\|_{C,\alpha Q}\chi_{Q}(x)
\end{align*}
where 
\[
\kappa_{\xi,\varepsilon,C}=2\xi\left(\frac{\varepsilon}{3c_{1}c_{2}}\right)+\xi\left(\frac{1}{3c_{1}}\right)\frac{\varepsilon}{3c_{1}c_{2}}\|M_{C}\|
\]
and $\xi(\rho)=\psi(\rho)+\phi(\rho)$ with $c_{1}$ and $c_{2}$
being constants depending on the parameters defining $\mathcal{D}$.
\end{lemma}

\begin{proof}
We shall use the following identity that was obtained in \cite[p. 684]{PTG}
\[
T_{{\bf b}}f(x)=\sum_{h=0}^{m}\sum_{\sigma\in C({\bf b})}(-1)^{m-h}\left(b(x)-{\bf \lambda}\right)_{\sigma'}T\left(\left(b-{\bf \lambda}\right)_{\sigma}f\right)(x)
\]
By the doubling condition of the measure there exists $c_{1}$ such
that $\mu(\alpha P)\leq c_{1}\mu(P)$ for any cube $P$. Now we observe
that for any disjoint family $\{P_{j}\}\subset\mathcal{D}(Q)$ 
\begin{align*}
 & T_{{\bf b}}(f\chi_{\alpha Q})(x)\chi_{Q}(x)\\
 & =T_{{\bf b}}(f\chi_{\alpha Q})(x)\chi_{Q\setminus\cup P_{j}}(x)+\sum_{j}T_{{\bf b}}(f\chi_{\alpha Q\setminus\alpha P_{j}})(x)\chi_{P_{j}}(x)+\sum_{j}T_{{\bf b}}(f\chi_{\alpha P_{j}})(x)\chi_{P_{j}}(x)
\end{align*}
and we have that 
\begin{align*}
 & T_{{\bf b}}(f\chi_{\alpha Q})(x)\chi_{Q}(x)-\sum_{j}T_{{\bf b}}(f\chi_{\alpha P_{j}})(x)\chi_{P_{j}}(x)\\
 & =T_{{\bf b}}(f\chi_{\alpha Q})(x)\chi_{Q\setminus\cup P_{j}}(x)+\sum_{j}T_{{\bf b}}(f\chi_{\alpha Q\setminus\alpha P_{j}})(x)\chi_{P_{j}}(x)\\
 & =\sum_{h=0}^{m}\sum_{\sigma\in C({\bf b})}(-1)^{m-h}\left(b(x)-c_{Q}\right)_{\sigma}T\left(\left(b-c_{Q}\right)_{\sigma'}f\chi_{\alpha Q}\right)(x)\chi_{Q\setminus\cup P_{j}}(x)\\
 & +\sum_{h=0}^{m}\sum_{\sigma\in C({\bf b})}\sum_{j}(-1)^{m-h}\left(b(x)-c_{Q}\right)_{\sigma}T\left(\left(b-c_{Q}\right)_{\sigma'}f\chi_{\alpha Q\setminus\alpha P_{j}}\right)(x)\chi_{P_{j}}(x)
\end{align*}
We claim that there exists a disjoint family $\{P_{j}\}\subset\mathcal{D}(Q)$
with 
\[
\sum_{j}\mu(P_{j})\leq\varepsilon\mu(Q)
\]
and 
\[
\begin{split} & |T_{{\bf b}}(f\chi_{\alpha Q})(x)\chi_{Q\setminus\cup P_{j}}+\sum_{j}T_{{\bf b}}(f\chi_{\alpha Q\setminus\alpha P_{j}})(x)\chi_{P_{j}}|\\
 & \leq\kappa_{n,\rho,s}\sum_{h=0}^{m}\sum_{\sigma\in C_{h}({\bf b})}|b(x)-c_{Q}|_{\sigma}\|f|b-c_{Q}|_{\sigma'}\|_{C,\alpha Q}\chi_{Q}(x)
\end{split}
\]
Note that to settle the claim, we are actually going to show that
\begin{align*}
 & \left|\sum_{h=0}^{m}\sum_{\sigma\in C({\bf b})}(-1)^{m-h}\left(b(x)-c_{Q}\right)_{\sigma}T\left(\left(b-c_{Q}\right)_{\sigma'}f\chi_{\alpha Q}\right)(x)\chi_{Q\setminus\cup P_{j}}(x)\right.\\
 & \left.+\sum_{h=0}^{m}\sum_{\sigma\in C({\bf b})}\sum_{j}(-1)^{m-h}\left(b(x)-c_{Q}\right)_{\sigma}T\left(\left(b-c_{Q}\right)_{\sigma'}f\chi_{\alpha Q\setminus\alpha P_{j}}\right)(x)\chi_{P_{j}}(x)\right|\\
 & \leq\kappa_{n,\rho,s}\sum_{h=0}^{m}\sum_{\sigma\in C_{h}({\bf b})}|b(x)-c_{Q}|_{\sigma}\|f|b-c_{Q}|_{\sigma'}\|_{C,\alpha Q}\chi_{Q}(x)
\end{align*}
For $\rho\in(0,1)$ to be chosen let 
\[
\mathcal{\tilde{M}}_{T}(f)=\max_{\sigma\in C_{h}({\bf b}),\,h=0,\dots,m}\left\{ \frac{|T(|b-c_{Q}|_{\sigma'}f\chi_{\alpha Q})|}{\xi(\rho)\||b-c_{Q}|_{\sigma'}f\|_{A,\alpha Q}},\frac{\mathcal{M}_{T,\alpha}^{\#}(|b-c_{Q}|_{\sigma'}f\chi_{\alpha Q})}{\xi(\rho)\||b-c_{Q}|_{\sigma'}f\|_{B,\alpha Q}}\right\} 
\]
and let us define the sets 
\[
\Omega=\left\{ x\in Q\,:\,\max\left\{ \frac{M_{C}(|b-c_{Q}|_{\sigma'}f\chi_{\alpha Q})(x)}{\rho\|M_{C}\|\||b-c_{Q}|_{\sigma'}f\|_{C,\alpha Q}},\mathcal{\tilde{M}}_{T}(f)\right\} >1\right\} 
\]
Observe that 
\begin{align*}
\begin{split}\mu(\Omega) & \leq\sum_{\sigma\in C_{h}({\bf b}),\,h=0,\dots,m}\mu\left(\left\{ x\in Q\,:\,\frac{M_{C}(|b-c_{Q}|_{\sigma'}f\chi_{\alpha Q})(x)}{\rho\|M_{C}\|\||b-c_{Q}|_{\sigma'}f\|_{C,\alpha Q}}>1\right\} \right)\\
 & +\sum_{\sigma\in C_{h}({\bf b}),\,h=0,\dots,m}\mu\left(\left\{ x\in Q\,:\,\frac{|T(|b-c_{Q}|_{\sigma'}f\chi_{\alpha Q})|}{\xi(\rho)\||b-c_{Q}|_{\sigma'}f\|_{A,\alpha Q}}>1\right\} \right)\\
 & +\sum_{\sigma\in C_{h}({\bf b}),\,h=0,\dots,m}\mu\left(\left\{ x\in Q\,:\,\frac{\mathcal{M}_{T,\alpha}^{\#}(|b-c_{Q}|_{\sigma'}f\chi_{\alpha Q})}{\xi(\rho)\||b-c_{Q}|_{\sigma'}f\|_{B,\alpha Q}}>1\right\} \right)\\
 & \leq3\Gamma_{m}\rho\mu(\alpha Q)\leq3\Gamma_{m}\rho c_{1}\mu(Q)
\end{split}
\end{align*}
where $\Gamma_{k}=\sum_{\sigma\in C_{h}({\bf b}),\,h=0,\dots,m}$.
Now we take the local Calderón-Zygmund decomposition (see \cite[Lemma 4.5]{FN})
of 
\[
\Omega_{c_{2}}=\left\{ s\in Q\,:\,M^{\mathcal{D}(Q)}(\chi_{\Omega})>\frac{1}{c_{2}}\right\} \qquad c_{2}\geq2.
\]
For a suitable choice of $c_{2}$ we have that $\Omega_{c_{2}}$ is
a proper subset of $Q$ and that there exists a family $\{P_{j}\}\subset\mathcal{D}(Q)$
such that $\Omega_{c_{2}}=\bigcup P_{j}$ and 
\[
\frac{1}{c_{2}}\leq\frac{\mu(P_{j}\cap\Omega)}{\mu(P_{j})}\leq\frac{1}{2}.
\]
Then we have that 
\[
\sum_{j}\mu(P_{j})\leq c_{2}\sum_{j}\mu(P_{j}\cap\Omega)\leq c_{2}\mu(\Omega)\leq3\rho c_{1}c_{2}\mu(Q),
\]
and choosing $\rho=\frac{\varepsilon}{3\Gamma_{m}c_{1}c_{2}}$, 
\[
\sum_{j}\mu(P_{j})\leq\varepsilon\mu(Q).
\]

Note that by the Lebesgue differentiation theorem there exists some
set $N$ of measure zero such that 
\begin{equation}
\Omega\setminus N\subset\Omega_{c_{2}}=\bigcup_{j}P_{j}\label{eq:Null-1}
\end{equation}
Now we show that this family $\{P_{j}\}$ is suitable for the claim
above to hold. Taking \eqref{eq:Null-1} into account if $x\in Q\setminus\cup P_{j}$
then the inequalities in $\Omega$ hold reversed a.e. in particular,
\[
\left|T(|b-c_{Q}|_{\sigma'}f\chi_{\alpha Q})\right|\leq\xi(\rho)\||b-c_{Q}|_{\sigma'}f\|_{A,\alpha Q}
\]
Now we deal with each term $T(f\chi_{\alpha Q\setminus\alpha P_{j}})(x)\chi_{P_{j}}(x)$.
First we note that $\mu(P_{j}\setminus\Omega)\not=0$. Indeed 
\[
\mu(P_{j})=\mu(P_{j}\cap\Omega)+\mu(P_{j}\setminus\Omega)\leq\frac{1}{2}\mu(P_{j})+\mu(P_{j}\setminus\Omega)
\]
Then for each $x'\in P_{j}\setminus\Omega$ 
\[
\begin{split} & \left|T\left(\left(b-c_{Q}\right)_{\sigma'}f\chi_{\alpha Q\setminus\alpha P_{j}}\right)(x)\right|\\
 & =\left|T\left(\left(b-c_{Q}\right)_{\sigma'}f\chi_{\alpha Q\setminus\alpha P_{j}}\right)(x)-T\left(\left(b-c_{Q}\right)_{\sigma'}f\chi_{\alpha Q\setminus\alpha P_{j}}\right)(x')\right|\\
 & +\left|T\left(\left(b-c_{Q}\right)_{\sigma'}f\chi_{\alpha Q\setminus\alpha P_{j}}\right)(x')\right|
\end{split}
\]
For $I$ we observe that since $x'\in P_{j}\setminus\Omega$
\begin{align*}
 & \left|T\left(\left(b-c_{Q}\right)_{\sigma'}f\chi_{\alpha Q\setminus\alpha P_{j}}\right)(x)-T\left(\left(b-c_{Q}\right)_{\sigma'}f\chi_{\alpha Q\setminus\alpha P_{j}}\right)(x')\right|\\
 & \leq\mathcal{M}_{T,\alpha}^{\#}(\left(b-c_{Q}\right)_{\sigma'}f\chi_{\alpha Q})(x')\leq\xi(\rho)\||b-c_{Q}|_{\sigma'}f\|_{A,\alpha Q}.
\end{align*}
 Then 
\[
\left|T\left(\left(b-c_{Q}\right)_{\sigma'}f\chi_{\alpha Q\setminus\alpha P_{j}}\right)(x)\right|\leq\xi(\rho)\||b-c_{Q}|_{\sigma'}f\|_{C,\alpha Q}+\inf_{x'\in P_{j}\setminus\Omega}\left|T\left(\left(b-c_{Q}\right)_{\sigma'}f\chi_{\alpha Q\setminus\alpha P_{j}}\right)(x')\right|
\]
For the remaining term we note that 
\[
\left|T\left(\left(b-c_{Q}\right)_{\sigma'}f\chi_{\alpha Q\setminus\alpha P_{j}}\right)(x')\right|\leq\left|T\left(\left(b-c_{Q}\right)_{\sigma'}f\chi_{\alpha Q}\right)(x')\right|+\left|T\left(\left(b-c_{Q}\right)_{\sigma'}f\chi_{\alpha P_{j}}\right)(x')\right|
\]
For the first term, as above 
\[
\left|T\left(\left(b-c_{Q}\right)_{\sigma'}f\chi_{\alpha Q}\right)(x')\right|\leq\xi(\rho)\||b-c_{Q}|_{\sigma'}f\|_{A,\alpha Q}.
\]
Then 
\[
\inf_{x'\in P_{j}\setminus\Omega}\left|T(\left(b-c_{Q}\right)_{\sigma'}f\chi_{\alpha Q\setminus\alpha P_{j}})(x')\right|\leq\xi(\rho)\||b-c_{Q}|_{\sigma'}f\|_{A,\alpha Q}+\inf_{x'\in P_{j}\setminus\Omega}\left|T\left(\left(b-c_{Q}\right)_{\sigma'}f\chi_{\alpha P_{j}}\right)(x')\right|
\]
Now we note that for $t=\frac{1}{3c_{1}}$ 
\[
\begin{split} & \mu\left(\left\{ x\in P_{j}\,:\,\left|T\left((b-c_{Q})_{\sigma'}f\chi_{\alpha P_{j}}\right)(x)\right|>\xi(t)\|(b-c_{Q})_{\sigma'}f\|_{A,\alpha P_{j}}\right\} \right)\\
 & \leq t\mu(\alpha P_{j})\leq\frac{c_{1}}{3c_{1}}\mu(P_{j})=\frac{1}{3}\mu(P_{j})
\end{split}
\]
Then, necessarily 
\[
\begin{split}\inf_{x'\in P_{j}\setminus\Omega}\left|T\left(\left(b-c_{Q}\right)_{\sigma'}f\chi_{\alpha P_{j}}\right)(x')\right| & \leq\xi(t)\|(b-c_{Q})_{\sigma'}f\|_{A,\alpha P_{j}}\\
 & \leq\xi(t)\inf_{x'\in P_{j}\setminus\Omega}M_{C}(\left(b-c_{Q}\right)_{\sigma'}f\chi_{\alpha Q})(x')\\
 & \leq\xi(t)\rho\|M_C\|\|\left(b-c_{Q}\right)_{\sigma'}f\|_{C,\alpha Q}
\end{split}
\]
Indeed. First we recall that since $\frac{\mu(P_{j}\cap\Omega)}{\mu(P_{j})}\leq\frac{1}{2}$,
then
\[
\mu(P_{j})=\mu(P_{j}\cap\Omega)+\mu(P_{j}\setminus\Omega)\leq\frac{1}{2}\mu(P_{j})+\mu(P_{j}\setminus\Omega)
\]
and we have that
\[
\frac{1}{2}\mu(P_{j})\leq\mu(P_{j}\setminus\Omega).
\]
Now we observe that if we had
\[
P_{j}\setminus\Omega\subset\left\{ x\in P_{j}\,:\,\left|T\left((b-c_{Q})_{\sigma'}f\chi_{\alpha P_{j}}\right)(x)\right|>\xi(\rho)\|\left(b-c_{Q}\right)_{\sigma'}f\|_{A,\alpha P_{j}}\right\} 
\]
then 
\[
\begin{split}\frac{1}{2}\mu(P_{j}) & \leq\mu(P_{j}\setminus\Omega)\\
 & \leq\mu\left(\left\{ x\in P_{j}\,:\,\left|T\left((b-c_{Q})_{\sigma'}f\chi_{\alpha P_{j}}\right)(x)\right|>\xi(\rho)\|\left(b-c_{Q}\right)_{\sigma'}f\|_{A,\alpha P_{j}}\right\} \right)\\
 & \leq\frac{1}{3}\mu(P_{j})
\end{split}
\]
which would be a contradiction. Gathering the estimates above the
claim holds and hence we are done.
\end{proof}
 Finally armed with the Lemma above we are in the position to settle
Theorem \ref{thm:ThmSparseComm}. 
\begin{proof}
[Proof of Theorem \ref{thm:ThmSparseComm}] Let $Q$ be a cube.
We iterate Lemma \ref{lem:It-1} and we get $\{Q_{j}^{l}\}$ families
of cubes with 
\[
\sum_{Q_{j}^{l+1}\subset Q_{i}^{l}}\mu(Q_{j}^{l+1})\leq\varepsilon\mu(Q_{i}^{l})
\]
such that 
\[
\begin{split}|T_{{\bf b}}(f\chi_{\alpha Q})(x)|\chi_{Q}(x) & \leq\kappa_{n,\rho,s}\sum_{h=0}^{m}\sum_{\sigma\in C_{h}({\bf b})}\sum_{l=0}^{L-1}\sum_{j}|b(x)-c_{Q_{j}^{l}}|_{\sigma}\|f|b-c_{Q_{j}^{l}}|_{\sigma'}\|_{C,\alpha Q_{j}^{l}}\chi_{Q_{j}^{l}}(x)\\
 & +\sum_{j}T_{{\bf b}}(f\chi_{\alpha Q_{j}^{L}})(x)\chi_{Q_{j}^{L}}(x).
\end{split}
\]
Note that 
\[
\sum_{j}\mu(Q_{j}^{L})\leq\varepsilon^{L}\mu(Q).
\]
Hence, letting $L\rightarrow\infty$ 
\[
|T_{{\bf b}}(f\chi_{\alpha Q})(x)|\chi_{Q}(x)\leq\kappa_{n,\rho,s}\sum_{h=0}^{m}\sum_{\sigma\in C_{h}({\bf b})}\sum_{l=0}^{\infty}\sum_{j}|b(x)-c_{Q_{j}^{l}}|_{\sigma}\|f|b-c_{Q_{j}^{l}}|_{\sigma'}\|_{C,\alpha Q_{j}^{l}}\chi_{Q_{j}^{l}}(x)
\]
and clearly 
\[
\mathcal{S}_{Q}=\bigcup_{j,l}\{Q_{j}^{l}\}
\]
is a $(1-\varepsilon)$-sparse family. Now we use Lemma \ref{lemma:covering}
with $E=\text{supp}(f)$, there exists a partition of $X$, $\mathcal{P}\subset\mathcal{D}$
such that $E\subseteq\alpha Q$ for every $Q\in\mathcal{P}$. Then
\[
T_{{\bf b}}f(x)=\sum_{Q\in\mathcal{D}}T_{{\bf b}}(f\chi_{\alpha Q})(x)\chi_{Q}(x)
\]
and it suffices to apply the estimate above to each term and we are
done just choosing in this case $c_{Q_{j}^{k}}=b_{Q_{j}^{k}}$. Note
that since $\mathcal{D}$ is a partition $\mathcal{S}=\cup_{Q\in\mathcal{D}}\mathcal{S}_{Q}$
is a $(1-\varepsilon)$-sparse family and hence we are done.\\
To prove the furthermore part, we fix the parameters in Proposition
\ref{proposition:dyadicsystem}. Then there exist $\mathcal{D}_{1},\dots,\mathcal{D}_{t_{0}}$
dyadic systems associated to those parameters. We repeat the argument
above for $\mathcal{D}_{1}$ and its parameters. Then there exists
a $(1-\varepsilon)$-sparse family $\mathcal{S}\subset\mathcal{D}_{1}$
such that 
\[
|T_{{\bf b}}f(x)|\lesssim\kappa_{\xi,\varepsilon,C}\sum_{h=0}^{m}\sum_{\sigma\in C_{h}({\bf b})}\sum_{Q\in\mathcal{S}}\||b-c_{Q}|_{\sigma'}f\|_{C,\alpha Q}|b-c_{Q}|_{\sigma}\chi_{Q}(x)
\]
Now by Proposition \ref{proposition:dyadicsystem}, we have that for
any $Q\in\mathcal{S}$ with center $z$ and side length $\delta^{k}$
we can find $Q'\in\mathcal{D}_{j}$ for some $1\leq j\leq m_{0}$
such that 
\[
\alpha Q=B(z,\alpha C_{0}\delta^{k})\subseteq Q'\qquad\text{diam}(Q')\leq\gamma\alpha C_{0}\delta^{k}
\]
then there exists $c>0$ depending on $X$ and $\alpha$ such that
\[
\mu(Q')\leq\mu(B(z,\alpha\gamma C_{0}\delta^{k}))\leq c\mu B(z,C_{0}\delta^{k})\leq c\mu(Q).
\]
Taking $E_{Q'}=E_{Q}$ we have that 
\[
\mu(Q')\leq c\mu(P)\leq\frac{c}{1-\varepsilon}\mu(E_{Q})=\frac{c}{1-\varepsilon}\mu(E_{Q'})
\]
and hence
\[
\mu(Q')\frac{1-\varepsilon}{c}\leq\mu(E_{Q'})
\]
and hence the collections of cubes 
\[
\tilde{\mathcal{S}_{j}}=\left\{ Q'\in\mathcal{D}_{j}\,:\,Q\in\mathcal{S}\right\} 
\]
are $\frac{1-\varepsilon}{c}$-sparse. Gathering the facts above,
at this point we have, choosing $c_{Q}=b_{Q'}$, the following
estimate
\[
|T_{{\bf b}}f(x)|\lesssim c\kappa_{\xi,\varepsilon,C}\sum_{h=0}^{k}\sum_{\sigma\in C_{h}({\bf b})}\sum_{j=1}^{t_{0}}\sum_{Q\in\tilde{\mathcal{S}_{j}}}\||b-b_{Q}|_{\sigma'}f\|_{C,\alpha Q}|b(x)-b_{Q}|_{\sigma}\chi_{Q}(x)
\]
and all we are left to do is to show that we can actually choose $(1-\varepsilon)$-sparse
families. For that purpose, observe that our families $\tilde{\mathcal{S}_{j}}$
are $\frac{c}{1-\varepsilon}$-Carleson. Let $u>c$ an integer. In
virtue of Lemma \ref{LemmaCarleson} we have that we can write $\tilde{\mathcal{S}_{j}}=\cup_{i=1}^u\mathcal{S}_{j}^{i}$
where each $\mathcal{S}_{j}^{i}$ is $1+\frac{\frac{c}{(1-\varepsilon)}-1}{u}$-Carleson.
Note that
\begin{align*}
1+\frac{\frac{c}{1-\varepsilon}-1}{u} &
=1+\frac{c-1+\varepsilon}{u(1-\varepsilon)}
 \leq1+\frac{\varepsilon}{1-\varepsilon}
 =\frac{1}{1-\varepsilon}.
\end{align*}
This yields that, each $\text{\ensuremath{\mathcal{S}_{j}^{i}} is \ensuremath{\frac{1}{1-\varepsilon}}}$-Carleson
and consequently, $(1-\varepsilon)$-sparse. Then we have that
\[
|T_{{\bf b}}f(x)|\lesssim c\kappa_{\xi,\varepsilon,C}\sum_{h=0}^{m}\sum_{\sigma\in C_{h}({\bf b})}\sum_{j=1}^{t_{0}}\sum_{i=1}^{u}\sum_{Q\in\mathcal{S}_{i}^{j}}\||b-b_{Q}|_{\sigma'}f\|_{C,\alpha Q}|b(x)-b_{Q}|_{\sigma}\chi_{Q}(x)
\]
and reindexing, we are done.
\end{proof}


\begin{thebibliography}{999}
\bibitem{B19}
Berra, Fabio. Mixed weak estimates of Sawyer type for generalized maximal operators. Proc. Amer. Math. Soc. 147 (2019), no. 10, 4259--4273
\bibitem{BCP19}
Berra, Fabio; Carena, Marilina; Pradolini, Gladis. Mixed weak estimates of Sawyer type for fractional integrals and some related operators. J. Math. Anal. Appl. 479 (2019), no. 2, 1490--1505
\bibitem{BCP21}
Berra, Fabio; Carena, Marilina; Pradolini, Gladis Mixed inequalities for commutators with multilinear symbol. Collect. Math. (2022). Available online.
\bibitem{CRR20}
Caldarelli, M., Rivera-Ríos, I.P. A sparse approach to mixed weak type inequalities. Math. Z. 296, 787–812 (2020). 
\bibitem{CW71}
Coifman, Ronald R.; Weiss, Guido. Analyse harmonique non-commutative sur certains espaces homogènes. (French) Étude de certaines intégrales singulières. Lecture Notes in Mathematics, Vol. 242. Springer-Verlag, Berlin-New York, 1971. {\rm v}+160 pp.
\bibitem{CUMP05}
Cruz-Uribe, D.; Martell, J. M.; Pérez, C. Weighted weak-type inequalities and a conjecture of Sawyer. Int. Math. Res. Not. 2005, no. 30, 1849--1871.
\bibitem{CUMT21}
Cruz-Uribe, D., Moen, K.,  Tran, Q. M. (2021). New oscillation classes and two weight bump conditions for commutators.
Collect. Math. (2022).
\bibitem{DGKLW}
 Duong, Xuan Thinh; Gong, Ruming; Kuffner, Marie-Jose S.; Li, Ji; Wick, Brett D.; Yang, Dongyong Two weight commutators on spaces of homogeneous type and applications. J. Geom. Anal. 31 (2021), no. 1, 980–1038.
\bibitem{FN}
Frey, D.; Nieraeth, Z.
Weak and strong type $A_1-A_\infty$ estimates for sparsely dominated operators. 
J. Geom. Anal. 29 (2019), no. 1, 247--282.
\bibitem{H12}
Hyt\"onen, Tuomas P. The sharp weighted bound for general Calderón-Zygmund operators. Ann. of Math. (2) 175 (2012), no. 3, 1473--1506.
\bibitem{HytNotes}
Hyt\"onen, Tuomas. Dyadic analysis and weights. Lecture Notes. University of Helsinki, Helsinki, 2017.
\bibitem{HK12}Hyt\"onen, Tuomas; Kairema, Anna. Systems of dyadic cubes in a doubling metric space. Colloq. Math. 126 (2012), no. 1, 1–33.
\bibitem{HPR12}
Hyt\"onen, Tuomas; Pérez, Carlos; Rela, Ezequiel.
Sharp reverse H\"older property for $A_\infty$ weights on spaces of homogeneous type. 
J. Funct. Anal. 263 (2012), no. 12, 3883–3899.
\bibitem{HLO}
Hytönen, T. Li, K.; Oikari, T.
Iterated commutators under a joint condition on the tuple of multiplying functions. (English summary)
Proc. Amer. Math. Soc. 148 (2020), no. 11, 4797--4815.
\bibitem{IFRR20}
Ibañez-Firnkorn, Gonzalo H.; Rivera-Ríos, Israel P. Sparse and weighted estimates for generalized H\"ormander operators and commutators. Monatsh. Math. 191 (2020), no. 1, 125--173.
\bibitem{L13}
Lerner, Andrei K. A simple proof of the $A_2$ conjecture. Int. Math. Res. Not. IMRN 2013, no. 14, 3159--3170.
\bibitem{LN}
Lerner, Andrei K.; Nazarov, Fedor Intuitive dyadic calculus: the basics. Expo. Math. 37 (2019), no. 3, 225--265.
\bibitem{LO}
Lerner, Andrei K.; Ombrosi, Sheldy Some remarks on the pointwise sparse domination. J. Geom. Anal. 30 (2020), no. 1, 1011--1027. 
\bibitem{LORR}
Lerner, Andrei K.; Ombrosi, Sheldy; Rivera-R\'ios, Israel P. On pointwise and weighted estimates for commutators of Calder\'on-Zygmund operators. Adv. Math. 319 (2017), 153--181.
\bibitem{LOP19}
Li, Kangwei; Ombrosi, Sheldy; Pérez, Carlos. Proof of an extension of E. Sawyer's conjecture about weighted mixed weak-type estimates. Math. Ann. 374 (2019), no. 1-2, 907--929.
\bibitem{LOPi19}
Li, Kangwei; Ombrosi, Sheldy J.; Belén Picardi, M. Weighted mixed weak-type inequalities for multilinear operators. Studia Math. 244 (2019), no. 2, 203--215
\bibitem{LMPR}
Lorente, Mar\'{\i}a; Martell, José María; P\'erez, Carlos; Riveros, Mar\'{\i}a Silvina Generalized Hörmander conditions and weighted endpoint estimates. Studia Math. 195 (2009), no. 2, 157–192.
\bibitem{LMRT}
Lorente, Mar\'{\i}a; Martell, José María; Riveros, María Silvina; de la Torre, Alberto. Generalized Hörmander's conditions, commutators and weights. J. Math. Anal. Appl. 342 (2008), no. 2, 1399--1425.
\bibitem{L21}Lorist, Emiel On pointwise $\ell^r$-sparse domination in a space of homogeneous type. J. Geom. Anal. 31 (2021), no. 9, 9366–9405.
\bibitem{MW77} Muckenhoupt, Benjamin; Wheeden, Richard L.
Some weighted weak-type inequalities for the Hardy-Littlewood maximal function and the Hilbert transform.
Indiana Univ. Math. J. 26 (1977), no. 5, 801--816.
\bibitem{OP16}
Ombrosi, Sheldy; Pérez, Carlos. Mixed weak type estimates: examples and counterexamples related to a problem of E. Sawyer. Colloq. Math. 145 (2016), no. 2, 259--272
\bibitem{OPR16}
Ombrosi, Sheldy; Pérez, Carlos; Recchi, Jorgelina. Quantitative weighted mixed weak-type inequalities for classical operators. Indiana Univ. Math. J. 65 (2016), no. 2, 615--640.
\bibitem{P95}
Pérez, C. On sufficient conditions for the boundedness of the Hardy-Littlewood maximal operator between weighted $L^p$-spaces with different weights. Proc. London Math. Soc. (3) 71 (1995), no. 1, 135--157
\bibitem{PTG}
P\'erez, C., and Trujillo-Gonz\'alez, R. Sharp weighted estimates for multilinear commutators. J. London Math. Soc. (2) 65, 3 (2002), 672--692.
\bibitem{RR18}
Rivera-Ríos, Israel P.
Quantitative weighted estimates for singular integrals and commutators
PhD thesis. (2018) Available in BIRD, BCAM's institutional repository.
\bibitem{S85}
Sawyer, E. A weighted weak type inequality for the maximal function. Proc. Amer. Math. Soc. 93 (1985), no. 4, 610--614. 
\end{thebibliography}
\end{document}